\documentclass[a4paper,11pt,leqno]{article}
\usepackage{amsfonts,amssymb,amsmath,amsthm}
\usepackage{fullpage}
\usepackage{epic,eepic}

\numberwithin{equation}{section}

\newtheorem{theorem}{Theorem}[section]
\newtheorem{proposition}[theorem]{Proposition}
\newtheorem{corollary}[theorem]{Corollary}
\newtheorem{lemma}[theorem]{Lemma}
\newtheorem{remark}[theorem]{Remark}
\theoremstyle{definition}
\newtheorem{definition}[theorem]{Definition}

\newcommand{\R}{{\mathbb R}}
\newcommand{\C}{{\mathcal C}}

\newcommand{\eps}{{\varepsilon}}

\newcommand{\n}{\nabla}
\renewcommand{\a}{\alpha}
\renewcommand{\b}{\beta}
\renewcommand{\d}{\delta}
\newcommand{\g}{\gamma}
\renewcommand{\l}{\lambda}
\newcommand{\vark}{\varkappa}
\newcommand{\var}{\varphi}
\newcommand{\pr}{^\prime}
\renewcommand{\O}{\Omega}

\renewcommand{\div}{\textrm{div}}
\newcommand{\sgn}{\operatorname{sgn}}
\newcommand{\ind}{\hbox{\rm 1\hskip -4.5pt 1}}
\newcommand{\gen}{\mathcal{L}}
\newcommand{\convl}{\mathcal{T}}

\def\XXint#1#2#3{{\setbox0=\hbox{$#1{#2#3}{\int}$}
\vcenter{\hbox{$#2#3$}}\kern-.5\wd0}}

\def\Xint#1{\mathchoice
{\XXint\displaystyle\textstyle{#1}}%
{\XXint\textstyle\scriptstyle{#1}}%
{\XXint\scriptstyle\scriptscriptstyle{#1}}%
{\XXint\scriptscriptstyle\scriptscriptstyle{#1}}%
\!\int}

\def\dashint{\Xint-}

\def\Xint#1{\mathchoice
{\XXint\displaystyle\textstyle{#1}}%
{\XXint\textstyle\scriptstyle{#1}}%
{\XXint\scriptstyle\scriptscriptstyle{#1}}%
{\XXint\scriptscriptstyle\scriptscriptstyle{#1}}%
\!\!\int\!\!\!\!\int}
\def\XXint#1#2#3{{\setbox0=\hbox{$#1{#2#3}{\int\int}$}
\vcenter{\hbox{$#2#3$}}\kern-0.26\wd0}}

\def\dashint{\Xint-\!\!\!\!\!\!\!-}

\begin{document}

\title{Singular solutions to  the heat equations with nonlinear absorption
and Hardy potentials}

\author{
{\large Vitali Liskevich}\\
\small Department of Mathematics\\
\small University of Swansea\\
\small Swansea SA2 8PP, UK\\
{\tt v.a.liskevich@swansea.ac.uk}\\
\and
{\large Andrey Shishkov}\\
\small Institute of Applied\\
\small Mathematics and Mechanics\\
\small Donetsk 83114,  Ukraine\\
{\tt shishkov@iamm.ac.donetsk.ua}
\and
{\large Zeev Sobol}\\
\small Department of Mathematics\\
\small University of Swansea\\
\small Swansea SA2 8PP, UK\\
{\tt z.sobol@swansea.ac.uk}
}

\date{}

\maketitle

\begin{abstract}
We study the existence and nonexistence of singular solutions
to the  equation $u_t-\Delta u -
\frac{\kappa}{|x|^2}u+|x|^\alpha u|u|^{p-1}=0$, $p>1$, in
$\R^N\times[0,\infty)$, $N\ge 3$, with a singularity at the point $(0,0)$,
that is, nonnegative solutions satisfying $u(x,0)=0$ for
$x\ne0$, assuming that $\a>-2$ and $\kappa<\left(\frac{N-2}2\right)^2$.
The problem is transferred to
the one for a weighted Laplace-Beltrami operator with a
non-linear absorbtion, absorbing the Hardy potential in the
weight. A classification of a singular solution to the
weighted problem either as a {\it source solution} with a
multiple of the Dirac mass as initial datum, or as a unique
{\it very singular solution}, leads to a complete
classification of singular solutions to the original
problem, which exist if and only if $p<1+\frac{2(2+\alpha)}{N+2+\sqrt{(N-2)^2-4\kappa}}$. 
\end{abstract}



\section{Introduction and main results}

In this paper we study
nontrivial nonnegative solutions in $\R^N\times[0,\infty)$ to the equation
\begin{equation}
\label{e1.1} u_t-\Delta u - \frac{\kappa}{r^2}u+r^\alpha u|u|^{p-1}=0, 
\end{equation}
vanishing on $\R^N\times\{0\}\setminus \{(0,0)\}$
that is, (non-trivial) nonnegative solutions to
\eqref{e1.1} satisfying
\begin{equation}
\label{initial-null}
u(x,0)=0,\mbox{ that is, }\lim_{t\to 0}u(x,t)=0 \ \text{for}\ x\ne0.
\end{equation}
Here and below $r=|x|$,
and we always assume that
with $N\ge 3$, $\kappa<\left(\frac{N-2}2\right)^2$ and
$\alpha>-2$.
The behaviour of $u(x,t)$ as $(x,t)\to (0,0)$ is not prescribed, so   we study solutions with possible singularity at $(0,0)$.

We remark here for further reference that there is no ambiguity in the last definition since for a
solution $u$ of \eqref{e1.1}, $u(x,t)dx\to0$ as $t\to0$ in the
sense of weak-$*$ convergence of measures on
$\mathbb{R}^N\setminus\{0\}$ if and only if $u(x,t)\to0$ as
$t\to0$ locally uniformly in $x\in \mathbb{R}^N\setminus\{0\}$,
by the same argument as in \cite[Proof of Theorem 2,
steps 2,3]{bf83}.

The nonlinear heat equation with absorption
\begin{equation}
\label{e1.1m}
u_t-\Delta u + u|u|^{p-1}=0,
\end{equation}
i.e. \eqref{e1.1} with $\kappa=\alpha=0$, with bounded measures
as initial data was first studied by Brezis and Friedman in the
seminal work \cite{bf83}, where it was proved that the solution
to \eqref{e1.1m} with $u(x,0)=\vark\delta_0(x)$, a multiple of
the Dirac mass at $0$, exists and is unique if and only if
$0<p<1 + \frac2N$. The solution obtained has roughly speaking
the same behaviour as $t\to 0$
as the fundamental solution to the linear heat equation. Such
solutions are referred to as {\it source type solutions} ({\bf
SS}).

In \cite{bpt86} for $1<p<1 + \frac2N$ a new nonlinear
phenomenon was discovered, namely,  a new solution to
\eqref{e1.1m} satisfying \eqref{initial-null} was found. This
solution is more singular at $t\to 0$ than the fundamental
solution, it is self-similar of the form
$t^{-\frac1{p-1}}f(|x|/\sqrt{t})$, where $f$ is a unique
solution to a certain ordinary differential equation. This
solution in \cite{bpt86}, called {\it very singular solution},
was constructed by the shooting method. Later in \cite{ek87}
the existence
of { very singular} solutions was
proved
by the
variational approach. In \cite{kp85} the very singular solution
was shown to be a monotone limit of source type solutions.
A classification of all positive singular solutions to \eqref{e1.1m}
was given in \cite{Os}.
The cited papers state that for
$p\in\big(1,1+\frac2N\big)$ every singular solution to
\eqref{e1.1} satisfying \eqref{initial-null} is either source
type solution, satisfying $u(x,t)dx\to \vark\delta_0$ as
$t\to0$ in the sense of weak convergence of measures, with
$\vark=\lim\limits_{t\to0}\int\limits_{\{|x|<1\}} u(x,t)dx$, or
$u$ is the unique { very singular} solution ({\bf{VSS}}), the
only one satisfying $\lim\limits_{t\to0}\int\limits_{\{|x|<1\}}
u(x,t)dx=\infty$. For $p\geq 1 + \frac2N$ there are no
non-trivial positive solutions to~\eqref{e1.1m} satisfying
\eqref{initial-null}. Recently the problem of singular
solutions to \eqref{e1.1} in the case $\kappa=0, \, \a>-2$ was
treated by Shishkov and Veron \cite{sv07}.  They showed that
the qualitative picture is the same as for \eqref{e1.1m}, but
the critical exponent changes from $1 + \frac2N$ for equation
\eqref{e1.1} to $1 + \frac{2+\a}N$ for the equation $u_t-\Delta
u +r^\alpha u|u|^{p-1}=0$. In all the above result a crucial
role is played by the following {\it a priori} estimates of
Keller-Osserman type for a singular solution to~\eqref{e1.1}
(with $\kappa =0$), which is a generalization of the classical
one due to Brezis and Friedman~\cite{bf83} in the case $\a=0$:
\begin{equation}
\label{KO-BF}
u(x,t)\leq c\left(|x|^2+t\right)^{-\frac{2+\alpha}{2(p-1)}}.
\end{equation}
The critical exponent $1 + \frac{2+\a}N$, which reflects the nonexistence of singular solutions, can be seen as a result
of comparing the behaviour at $(0,0)$ of the fundamental solution to the linear problem with estimate \eqref{KO-BF}.
This plausible argument can no longer be applied to \eqref{e1.1} since the fundamental solution to the equation
$ u_t-\Delta u - \frac{\kappa}{r^2}u=0$ does not exist at $x=0$ (except for the case $\kappa=0$) for the reasons explained below.

During the last decades there is growing interest to the
elliptic and parabolic problems involving the inverse-square
potential (Hardy potential), stemming from its criticality.
Many qualitative properties of solutions are affected by the
presence of the Hardy potential, which leads to occurrence of a
number of interesting unusual phenomena \cite{APP,BG,MRT,VZ}. This is mainly
due to the properties of the corresponding linear equation
$u_t-\Delta u - \frac{\kappa}{r^2}u=0$, which are significantly
different from the properties of the heat equation. In
particular, the linear equation does not have the fundamental
solution at zero, i.e. the Cauchy problem with $\d_0(x)$ as the
initial datum has no solution, which can be seen from by now
well known two-sided estimates for the corresponding heat
kernel $p(t,x,y)$\cite{LS,MS,MT}:
\[
\frac{c_1}{t^{\frac N2}} \left(\frac{|x|}{|x|+\sqrt{t}}\right)^\l  \left(\frac{|x|}{|x|+\sqrt{t}}\right)^\l   e^{-\frac{|x-y|^2}{c_2t}}\le     p(t,x,y)
\le \frac{c_3}{t^{\frac N2}} \left(\frac{|x|}{|x|+\sqrt{t}}\right)^\l \left(\frac{|y|}{|y|+\sqrt{t}}\right)^\l      e^{-\frac{|x-y|^2}{c_4t}},
\]
where here and below
$\lambda=-\frac{N-2}{2}+\sqrt{\left(\frac{N-2}{2}\right)^2-\kappa}$
is the bigger root of the quadratic equation
$\lambda^2+\lambda(N-2)+\kappa=0$ and $c_1,c_2,c_3,c_4$ are
some positive constants. Moreover, for $\kappa>0$ and $N\ge 3$
the Cauchy problem is not well posed in $L^p(\R^N)$, $p\in
\big[1, \frac{N}{N+\lambda}\big)$ \cite{lsv02}. The semilinear
equations with Hardy potentials and nonlinear excitation were
recently studied in \cite{MRT} and some interesting
nonuniqueness phenomena were discovered, but to our knowledge
the corresponding equation with nonlinear absorption, equation
\eqref{e1.1}, has not been yet studied. This is exactly the aim
the present paper. In the course of this study we will reveal
several interesting phenomena peculiar to equation
\eqref{e1.1}. In order to overcome the difficulties described
above and to classify solutions to \eqref{e1.1} satisfying
\eqref{initial-null} we will use the technique of transference
to the weighted space, which is by now standard in the linear
theory and is called the {\it ground state} transform (cf.
\cite{LS,MRT,MS,MT}). We will outline it here.
\bigskip

Below and further on we use the following notation for the weighted Lebesgue and Sobolev
spaces.
For a weight $\var$ we denote
\begin{equation*}
\begin{split}
&L^p_\var(\R^N):=\{f:\R^N\to \R;\ \int_{\R^N} |f|^p\var dx<\infty\},\\
&H^1_\var(\R^N):=\{f:\R^N\to \R;\ \int_{\R^N}(|\n f|^2+|f|^2)\var dx<\infty\}.
\end{split}
\end{equation*}
Let $h\in H^1_{loc}$ satisfy $\Delta h + \frac \kappa{r^2}h =
0$, that is, $h=r^\lambda$, with $\lambda> -\frac{N-2}2$ as
above.
The change of variables $\tilde u:= u/h$ is a unitary
operator $L^2:=L^2(\mathbb{R}^N,dx)\to L^2_{h^2}:=L^2(\mathbb{R}^N,
h^2dx)$. Moreover, the quadratic form $\mathcal{E}(u)=\int
|\nabla u|^2dx - \int \frac \kappa{r^2}u^2 dx$ on $L^2$ is
isomorphic to a quadratic form $\mathcal{E}_h(\tilde u)=\int
|\nabla \tilde u|^2 h^2 dx$ on $L^2_{h^2}$, which is precisely stated in the next proposition

\begin{proposition}
Let $\kappa<\frac{(N-2)^2}4$. Let $\mathcal{E}$ be the closed quadratic form in $L^2$ defined by
\[
\mathcal{E}(u)=\int |\n u|^2dx - \int \frac{\kappa}{r^2}|u|^2dx, \quad u\in H^1(\mathbb{R}^N)
\]
and $H$ be the associated
self-adjoint operator, $H=-\Delta-\frac\kappa{r^2}$ (form-sum). Let $h\in
H^1_{\mathrm{loc}}(\mathbb{R}^N)$ be a positive weak solution
to the equation $Hh=0$.

Then the unitary map $U:\,L^2\to L^2_{h^2}$, $Uu=\frac uh$,
maps $H$ to the operator $-\Delta_{h^2}$ associated with the
form
\[
\mathcal{E}_h(u)=\|\n u\|^2_{L^2_{h^2}} ,~u\in H^1_{h^2}(\mathbb{R}^N).
\]

\end{proposition}

\begin{proof}
First observe that $h\in C^\infty(\mathbb{R}^N\setminus \{0\})$
and that $h>0$.
Hence $h^{\pm 1}C_c^\infty(\mathbb{R}^N\setminus
\{0\})=C_c^\infty(\mathbb{R}^N\setminus \{0\})$. Note that $C_c^\infty(\mathbb{R}^N\setminus \{0\})$ is a core of the form $\mathcal{E}$.
The image of $\mathcal{E}$ on $L^2_{h^2}$ is given by $\mathcal{E}_h(\varphi)=\mathcal{E}(h\varphi)$.
%
For $\phi\in
C_c^\infty(\mathbb{R}^N\setminus \{0\})$ one has

\[
|\nabla(h\phi)|^2 = h^2|\nabla \phi|^2 + 2h\phi\nabla\phi\nabla h + \phi^2|\nabla h|^2
= h^2|\nabla \phi|^2 + \nabla(h\phi^2)\nabla h.
\]
Taking into account that $h$ is a weak solution to the equation
$\Delta u+\frac{\kappa}{r^2}u=0$, we obtain
\[
\mathcal{E}(h\phi)=\int \left(|\n(h\phi)|^2-\frac{\kappa}{r^2}h^2\phi^2\right)dx
=\int h^2|\n \phi|^2dx+\left(\n(h\phi^2)\cdot \n h-\frac{\kappa}{r^2}(h\phi^2)h\right)dx=\|\phi\|^2_{L^2_{h^2}}.
\]
Since $C_c^\infty(\mathbb{R}^N\setminus \{0\})$ is invariant under multiplication by $h^{\pm 1}$, it is also a core for
$\mathcal{E}_h$. The assertion follows.
\end{proof}
As a result, the operator $-\Delta - \frac{\kappa}{r^2}$ on $L^2$ is
isomorphic to the weighted Laplacian $-\Delta_{h^2}:= -\frac1{h^2}\nabla\cdot
h^2\nabla$ on $L^2_{h^2}$. Recall that $h=r^\l$.
So equation \eqref{e1.1} takes the form
\begin{equation}
\label{e1.1a}
\tilde u_t- r^{-2\lambda}\nabla (r^{2\lambda} \nabla \tilde u)+r^{\beta}\tilde u|\tilde u|^{p-1}=0,
\end{equation}
with $\beta:=\alpha + \lambda(p-1)$.

This motivates the following problem about singular solutions with the weighted Laplacian
which is of independent interest.

Assume that $u$ is a weak   solution to the equation
\begin{equation}
\label{wparabolic}
    \partial_t u - r^{-2\lambda}\nabla\cdot (r^{2\lambda}\nabla u) + r^\beta |u|^{p-1}u=0,
\end{equation}
satisfying
\begin{equation}\label{null}
    \int\limits_{\mathbb{R}^N}u(t)\theta\,h^2dx\to 0\quad \text{as}\ t\to 0,\mbox{ for all }\theta\in C_c(\mathbb{R}^N\setminus\{0\}).
\end{equation}
We say that $u\in L^2_{loc}\Big((0,\infty);\, H^{1}_{loc}(\mathbb{R}^N,h^2dx)\Big)\cap
L^{p+1}_{loc}\Big(\mathbb{R}^N\times(0,\infty),r^\beta h^2dxdt\Big)$ is a weak solution to \eqref{wparabolic}
if it satisfies
the integral identity
\begin{eqnarray}
\nonumber
    &&\int u(t_1)\zeta(t_1)h^2dx + \int\limits_{t_0}^{t_1}\int (\nabla u\cdot \nabla \zeta)h^2dx\,dt
    + \int\limits_{t_0}^{t_1}\int r^\beta |u|^{p-1}u\zeta h^2dx\,dt\\
&& =  \int u(t_0)\zeta(t_0)h^2dx + \int\limits_{t_0}^{t_1}\int u \partial_t\zeta\, h^2dx\,dt\,
\end{eqnarray}
for all $\zeta \in C^{1}\big((0,\infty);\; C_c(\mathbb{R}^N)\big)\cap L^2_{loc}\big((0,\infty);\; H^1_{h^2}\big)$
 and $0<t_0<t_1<\infty$.

The main results concerning the solutions to \eqref{wparabolic} satisfying \eqref{null} are collected in the following theorem.

\begin{theorem}
\label{main1}
Let $p>1$. Let $\l,\b$ be any real numbers such that $\lambda>-\frac{N-2}2$,  $\beta>-2$. Denote
$p^{*}=1+\frac{2+\beta}{N + 2\lambda}$.
Then
\begin{enumerate}
      \item[(a)] for any weak solution $u$ to \eqref{wparabolic} satisfying \eqref{null}  the following
      Keller-Osserman type estimate holds: there exists $c>0$ such that for all $x\in \R^N$ and $t>0$
      $$|u(x,t)|\leq
      c\left(|x|^2+t\right)^{-\frac{2+\beta}{2(p-1)}};$$
      \item[(b)] for every singular solution to
          \eqref{wparabolic} there exists
          $\vark\in[0,\infty]$ such that $u(t)h^2dx\to
          \vark\delta_0$ as $t\to0$ in the weak-$*$
          topology of Radon measures;
      \item[(c)] for $p\geq p^*$, the only solution to \eqref{wparabolic} satisfying \eqref{null} is zero;
      \item[(d)] for $p<p^*$ and $\vark\in(0,\infty]$ there
          exists a unique singular solution $u_\vark$ to
          \eqref{wparabolic} satisfying
$u_\vark(t)h^2dx\to\vark\delta_0$ as $t\to0$ in the
weak-$*$ topology of Radon measures; 
\item[(e)] for $p<p^*$ and $\vark\in(0,\infty)$,
    $u_\vark$
    satisfies $u_\vark(\cdot,t)-\vark
    p^{h^2}(t,\cdot,0)\to0$ in
    $L^1\left(\mathbb{R}^N,r^{2\lambda}dx\right)$ as
    $t\to0$, where $p^{h^2}$ is the heat kernel for the
    weighted Laplacian
    $\Delta_{h^2}:=r^{-2\lambda}\nabla\cdot
    r^{2\lambda}\nabla $;
      \item[(f)] for $p<p^*$, 
      $u_\infty$
      is self-similar, that is $u_\infty(x,t)=t^{-\frac{2+\beta}{2(p-1)}}v(x/\sqrt
      t)$, with $v(x)\leq Ce^{-|x|^2/8}$ for some $C>0$.
    \end{enumerate}
\end{theorem}

The proof of Theorem~\ref{main1} is given in Sections~2--5.
In fact, the results for the case of the initial data $\vark \d$ are obtained as a special
case of Radon measures as initial data, as it is done in \cite{bf83}. We extend the results
of Brezis and Friedman \cite{bf83} and V\'{e}ron~\cite[Chapter~6, Theorem~6.12]{Ver} to the case of equations
with the generator of a symmetric ultracontractive Feller semigroup in place of the Laplacian.
This allows for a much wider range of applications such as the fractional Laplacian, symmetric subelliptic operators
and many more (for further examples see, e.g. \cite{FOT}).

The solution to the original problem \eqref{e1.1}, \eqref{initial-null} is contained in the next corollary, 
which is a pull back of Theorem~\ref{main1}.

\begin{corollary}
Let $p>1$, $\kappa< \left(\frac{N-2}2\right)^2$, $\lambda=
-\frac{N-2}2 + \sqrt{\left(\frac{N-2}{2}\right)^2-\kappa}$,
$\alpha>-2$. Denote $p^{**}=1+\frac{2+\alpha}{N+\lambda}$.
Then
\begin{enumerate}
      \item[a)] for $p\geq p^{**}$, there are no singular solutions to
      \eqref{e1.1}. More precisely, the only solution to \eqref{e1.1} satisfying \eqref{initial-null} is zero;
      \item[(b)] for $p<p^{**}$ and $\vark\in(0,\infty]$, there exists a unique  singular
solution $u_\vark$ to \eqref{e1.1}
 satisfying
$\lim\limits_{t\to0}\int\limits_{\{|x|<\rho\}}
u_\vark(x,t)|x|^{\lambda}dx=\vark$ for all $\rho>0$. The map $\vark\to u_\vark$ is a bijection between
$(0,\infty]$ and the set of nontrivial singular solutions to \eqref{e1.1};
      \item[(c)] for $p<p^{**}$, the very singular solution $u_\infty$
      is self-similar, $u_\infty(x,t)=t^{-\frac{2+\alpha}{2(p-1)}}v(x/\sqrt
      t)$ with $v(x)\leq C|x|^\lambda e^{-|x|^2/8}$ for some $C>0$.
    \end{enumerate}
\end{corollary}

\begin{remark}
\label{lebesgue}
The above corollary shows that the Lebesgue measure does
not allow for a classification of singular solutions to
\eqref{e1.1}. To demonstrate this
let $\alpha$, $\kappa$, $\lambda$ and $p^{**}$ be as in the
preceding corollary. 
\begin{enumerate}
\item For $\kappa<0$ (hence $\lambda>0$) and
      $p\in(1,p^{**})$ every non-trivial positive singular
      solution $u$ to \eqref{e1.1} satisfies $u(x,t)=O(|x|^\lambda)$ as $x\to0$ for all $t>0$, and
      $\int\limits_{\{|x|<\rho\}} u(x,t)dx \to \infty$ as
      $t\to0$ for all $\rho>0$.
  \item for $\kappa>0$ (hence $\lambda<0$) and
      $p\in(1,p^{**})$, given $\vark\in(0,\infty)$, one has  $\int\limits_{\{|x|<\rho\}} u_\vark(x,t)dx \to 0$
      as $t\to0$ for all $\rho>0$. Moreover, $\int\limits_{\{|x|<\rho\}} u_\infty(x,t)dx
      \to 0$ as $t\to0$ for all $\rho>0$ if $p\in
      (1+\frac{2+\alpha}{N},p^{**})$. So in this case we have the initial datum zero with nonzero solution,
      and we encounter the {\bf non-uniqueness} phenomenon.
  \item For $\kappa>0$ (hence $\lambda<0$) one has $\int\limits_{\{|x|<\rho\}}
      u_\infty(x,t)dx \to \infty$ for all $\rho>0$ if $p\in
      (1,1+\frac{2+\alpha}{N})$.
  \item For $\kappa>0$ (hence $\lambda<0$) and $\rho>0$ one has $\int\limits_{\{|x|<\rho\}}
      u_\infty(x,t)dx \to c<\infty$ for $p= 1+\frac{2+\alpha}{N}$. The limit $c$ is independent of $\rho$.
      So this is the only case with the initial datum $c\d_0$.
\end{enumerate}
\end{remark}

Further on we use the following notation. For $p\in (1,\infty)$, $p'$ is the conjugate exponent, that is
$
p'=\frac{p}{p-1}
$.
$\ind_X$ stands for the characteristic function of the set $X$,
$B_R:=\{x\in \R^N: |x|\le R\}$,  $D_R:=\{(x,t)\in \R^N\times
(0,\infty): R^2\le |x|^2 + t\le 4R_0^2\}.$

For $\delta>0$, let $T_\delta$ denote the Steklov average,
$T_\delta u(t) =
\frac1{2\delta}\int\limits_{t-\delta}^{t+\delta}u(s)ds =
\frac1{2\delta}\int\limits_{-\delta}^{\delta}u(s+t)ds$.

We finish this section with the proposition classifying singular solutions to \eqref{wparabolic} i.e, the solutions satisfying
\eqref{null}. This is an analogue of \cite[Lemma 1.1]{kpv89}.

\begin{proposition}
\label{classify}
Let
$
u\in L^2_{loc}\Big((0,\infty);\,
H^{1}_{loc}(\mathbb{R}^N,h^2dx)\Big)\cap L^{p+1}_{loc}\Big(\mathbb{R}^N\times(0,\infty),r^\beta h^2dxdt\Big)
$
be a non-trivial
positive solution to \eqref{wparabolic} satisfying
\eqref{null}.
Then, for any $\rho>0$, there exists the limit
\[
\lim\limits_{t\to0}\int\limits_{B_\rho} u(x,t)|x|^{2\lambda}dx=:\vark\leq+\infty.
\]
The limit is independent of $\rho>0$.
\end{proposition}

For the proof we need the following lemma which will also be used further on.

\begin{lemma}\label{local-quali}
Let $u\in L^2_{loc}\Big((0,\infty);\,
H^{1}_{loc}(\mathbb{R}^N,h^2dx)\Big)\cap
L^{p+1}_{loc}\Big(\mathbb{R}^N\times(0,\infty),r^\beta
h^2dxdt\Big)$ be a ~solution (sub-solution) to \eqref{wparabolic} satisfying \eqref{null}.
Let $\tilde u$ denote the continuation of $u$ into the
semi-space $\mathbb{R}^N\times(-\infty,0)$ by zero.
Then, for every
domain $\Omega$ such that $\overline
\Omega \Subset \mathbb{R}^N\setminus\{0\}$, the function
$\tilde u$ is a solution (sub-solution) to \eqref{wparabolic} in
$\Omega\times \mathbb{R}$ and, moreover, $u\in
L^\infty_{loc}\Big(\mathbb{R}^N\times(0,\infty)\Big)$.

In particular, if $u$ is a solution to \eqref{wparabolic} then
$\tilde u\in C^{2,1}(\Omega\times \mathbb{R})$ and $u(x,t)\to
0$ as $t\to0$ uniformly in $x\in\Omega$.
\end{lemma}

\begin{proof}

Given $\Omega$ such that $\overline \Omega \Subset
\mathbb{R}^N\setminus\{0\}$,  observe that $h,r^\beta \in
C^{\infty}(\Omega)$ and there exists a constant $c>1$ such that
$\frac1c< h^2, r^\beta < c$ on $\Omega$. Hence the first
assertion follows from \cite[Proof of Theorem~2, steps~2,3]{bf83}.

To prove the second assertion, consider a cylinder
$B_R\times(t_0,t_1)$, $R>0$, $0<t_0<t_1$. Then there exists
$\tau\in (0,\frac12t_0)$ such that $u(\tau)\in
H^1_{h^2}(B_{2R})$. Moreover, $u$ is bounded on $\partial
B_{2R}\times (\tau,2t_1)$, by the first assertion. Let the
function $w$ be the solution to the problem
\[
\begin{cases}
\partial_t w - h^{-2}\nabla\cdot(h^2\nabla w) = 0 & \mbox{ in } B_{2R}\times (\tau,2t_1),\\
w(x,t) = u (x,t), & (x,t)\in B_{2R}\times\{\tau\}\cup \partial B_{2R}\times (\tau,2t_1).
\end{cases}
\]
Then $w$ is bounded on $B_R\times(t_0,t_1)$ \cite{gs-c05} and, by the maximum
principle, $|u|\leq |w|.$
\end{proof}

\noindent{\it Proof of Proposition~\ref{classify}}. 
First we show that if the limit exists, then it is independent of $\rho$.
Indeed, for $R>\rho$,
\[
\lim\limits_{t\to0}\int\limits_{B_R\setminus B_\rho} u\, h^2dx =0
\]
since $u(x,t)\to0$ as $t\to0$ uniformly in $x \in B_R\setminus
B_\rho$, by Lemma~\ref{local-quali}.

Now we show the existence of the limit. Note that $u=(u-1)^+ +
u\wedge 1$. Given $\rho>0$, Lemma~\ref{local-quali} implies
that there exists $T_\rho>0$ such that $u(x,t)<1$ for all $x\in
B_\rho\setminus B_{\rho/2}$, $t\in[0,T_\rho]$. Hence
$u_1:=(u-1)^+\ind_{B_\rho}\in L^2_{loc}\Big((0,T_\rho);\,
H^{1}_{h^2}(\mathbb{R}^N)\Big)$.

Next we integrate \eqref{wparabolic} over the set $\{u_1>0\}$.
To do this, consider the sequence $(\xi_n)_n$, $\xi_n:\R^+\to \R^+$, $\xi_n(s):=(ns)^+\wedge1$.
Then $(\xi_n)_n$ is a sequence of bounded non--decreasing Lipschitz
functions approximating $\ind_{(0,\infty)}$, so that $\xi_n(u_1)$
can be used as a test function for \eqref{wparabolic}. For
$0<s<t<\infty$ we have
\[
\int_s^t \int_{B_\rho}\xi_n(u_1) \partial_t u\,h^2dx\,dt = -\int_s^t \int_{B_\rho}\nabla\xi_n(u_1)\cdot \nabla u\,h^2dx\,dt
- \int_s^t \int_{B_\rho}\xi(u_1) u^pr^{\beta}h^2dx\,dt.
\]
Since $\nabla\xi_n(u_1)\cdot \nabla u \ge 0$,  it follows that $\int_s^t \int_{B_\rho}\xi_n(u_1) \partial_t u
r^{2\lambda}dx\,dt\leq 0.$ Note that $\xi_n(u_1)
\partial_t u = \partial_t \Xi_n(u_1)$, where $\Xi_n(s)=\int_{0}^s\xi_n(\tau)d\tau \to s^+$ as $n\to\infty$. 
Then
\[
\int_{B_\rho}\Xi_n(u_1)(t) r^{2\lambda}dx \leq \int_{B_\rho}\Xi_n(u_1)(s) r^{2\lambda}dx.
\]
Passing to the limit as $n\to\infty$ we obtain that
\[
\int_{B_\rho}(u-1)^+(t) h^2dx \leq \int_{B_\rho}(u-1)^+(s) h^2dx.
\]
Due to this monotonicity,
\[
\int_{B_\rho}(u-1)^+(t) h^2dx \to \vark\leq+\infty\mbox{ as }t\to0.
\]
Finally, $\int_{B_\rho}u\wedge 1\, h^2dx\to0$ by the Lebesgue
dominated convergence theorem. \qed

\begin{remark}
If $\int_{B_\rho}u(t)r^{2\lambda}dx \to \vark<\infty$ as $t\to0$ then
$u(t)h^2dx\to \vark\delta_0$ as $t\to0$ in the weak-$*$ topology of
Radon measures. Indeed, for any $\theta\in C_c(\mathbb{R}^N)$
there exists $R>0$ such that $\mathrm{supp}\,\theta\subset
B_R$, and, for any $\epsilon>0$ there exists $\rho>0$ such that
$|\theta(x)-\theta(0)|<\epsilon$ for all $x\in B_\rho$.
Then
\begin{equation}
\label{eq-rem}
\int_{\mathbb{R}^N}\theta u(t)h^{2}dx
= \theta(0)\int_{B_R}u(t)h^{2}dx
+ \int_{B_R\setminus B_\rho}(\theta -\theta(0))u(t)h^{2}dx
+ \int_{B_\rho}(\theta -\theta(0))u(t)h^{2}dx.
\end{equation}
Now $\int_{B_R\setminus B_\rho}(\theta -\theta(0))u(t)h^{2}dx\to0$
as $t\to0$ since $u(x,t)\to0$ as $t\to0$ uniformly in $x\in
B_R\setminus B_\rho$ and
\[
\limsup\limits_{t\to0}\left|\int_{B_\rho}(\theta -\theta(0))u(t)h^{2}dx\right|
\leq \epsilon \vark.
\]
Therefore it follows from \eqref{eq-rem} that, for any $\epsilon>0$,
\[
\limsup\limits_{t\to0}\left|\int_{\mathbb{R}^N}\theta u(t)h^{2}dx - \theta(0)\vark\right|< \epsilon \vark.
\]
\end{remark}

\bigskip
Proposition~\ref{classify} gives rise to the following definition giving classification to singular solutions to \eqref{wparabolic}.
\begin{definition}\label{ss-vss}
A non-trivial positive solution $u$ to \eqref{wparabolic}
satisfying \eqref{null} is called a {\it source-type solution} (SS) if
$\int_{B_\rho} u(t)r^{2\lambda}dx\to \vark$ with some finite $\vark>0$.
The solution $u$ is called a {\it very singular solution} (VSS) if
$\int_{B_\rho} u(t)r^{2\lambda}dx\to \infty$.
\end{definition}

The rest of the paper
is organized as follows.
In Section~\ref{sec2} we prove a-priory estimates of Keller-Osserman type and show that in the critical and supercritical range of values of $p$
the only solution to equation~\eqref{wparabolic} satisfying \eqref{null} is zero. In Section~3 we study general linear inhomogeneous  evolution
equations with a generator of a Feller semigroup and with Radon measures in the right hand side and as initial data. These results are applied in Section~4, where the
general semilinear equations with Radon measures as initial data are studied. The results are then applied to equation~\eqref{wparabolic}. Very singular solutions
to equation~\eqref{wparabolic} are discussed in Section~5.
Finally, in Appendix we give a version of the Hardy inequality and provide an auxiliary compactness result.

\section{A-priori estimates and nonexistence result}
\label{sec2}

We start with a-priori estimates for sub-solutions to \eqref{wparabolic} similar to that obtained in \cite{bf83}.
%
We use the notation
\[
D_{\rho,R}:=\{(x,t)\in \R^N\times \R_+\,:\, \rho^2<|x|^2+t<R^2\}.
\]

\begin{proposition}\label{local-quanti}
Let $0<R_0<\frac 14 R_1$ and $u$ be a sub-solution to
\eqref{wparabolic} in the paraboloid layer $D_{R_0,R_1}$ such that $u(x,t)\to u_0$ as $t\to0$ in
$L^2_{h^2}\big(B_{R_1}\setminus B_{R_0}\big)$. Then, for all
$R_0<\rho<\frac12 R<\frac12R_1$, 
\begin{equation}
\label{int-KO}
\iint\limits_{D_{2\rho,R}} \left[|\nabla u|^2 + r^\beta u^{p+ 1}\right]h^2 dx\,dt
\leq c \left(\rho^{N+2\lambda - 2\frac{2+\beta}{p-1}} + R^{N+2\lambda - 2\frac{2+\beta}{p-1}} + \|u_0\|_{L^2_{h^2}}^2\right).
\end{equation}
Moreover,  for $u_0=0$ one has, $u(x,t)\le c(|x|^2 +t)^{-\frac{2+\beta}{2(p-1)}}$ for
$4R_0^2<|x|^2 +t<R_1^2$.
\end{proposition}
\begin{proof} In the proof we use some of ideas from \cite{chist}.
Let $\phi\in C^1(\mathbb{R})$,
$\ind_{(4,\infty)}\le\phi\le\ind_{(1,\infty)}$, $|\phi^\prime|\le c
\phi^\alpha$, with $\alpha <1$ to be chosen later.
Let
\[\xi= \phi\left(\frac{|x|^2+t}{\rho^2}\right), \eta=
\phi\left(5-\frac{|x|^2+t}{R^2}\right) \mbox{ and } \zeta=\xi\eta.\]
Then
\begin{equation}\label{zeta}
    \begin{split}
\ind_{D_{2\rho,R}}& \le \zeta \le \ind_{D_{\rho,2R}},\\
|\partial_t\zeta|& \le \frac c{\rho^2}\ind_{D_{\rho,2\rho}}\xi^\alpha\eta + \frac c{R^2}\ind_{D_{R,2R}}\xi\eta^\alpha,\\
|\nabla\zeta| &\le \frac c{\rho}\ind_{D_{\rho,2\rho}}\xi^\alpha\eta + \frac c{R}\ind_{D_{R,2R}}\xi\eta^\alpha.
    \end{split}
\end{equation}
We set $u(x,t):=u_0(x)$ for $t\le0$,
$x\in B_{R_1}$ and choose
$T_\delta(\zeta^2 (T_\delta u))$ as a test function in
\eqref{wparabolic} on $D_{\rho,2R}$. Note that it
is a legitimate test function, by Lemma \ref{local-quali}.
Further on we denote $w=T_\delta u$. Then we obtain
\[\begin{split}
\iint|\nabla w \zeta|^2 h^2dxdt & + \iint r^\beta
T_\delta(|u|^{p-1}u)\zeta w\zeta h^2dxdt\\
\le & \iint w^{2}\zeta\partial_t\zeta\, h^2dxdt + \iint w^{2}|\nabla\zeta|^2\, h^2dxdt
+ \int\limits_{B_{2R}\setminus B_\rho}w^2(0)\zeta^2(0)h^2dx.
\end{split}\]
Passing to the limit as $\delta\to0$ in the last inequality, we
may replace $w$ with $u$. Now we estimate the first
two integrals in the right hand side using \eqref{zeta}. By the
Young inequality, for all $\varepsilon>0$ there exists
$c_\eps>0$ such that
\[\begin{split}
u^{2}\zeta|\partial_t\zeta|\le & u^{2}\zeta\left[\frac
c{\rho^2}\ind_{D_{\rho,2\rho}}\xi^\alpha\eta
+ \frac c{R^2}\ind_{D_{R,2R}}\xi\eta^\alpha\right]\\
\le & \varepsilon \left(r^{-2}u^2\zeta^2 +  r^\beta |u|^{p+1}\zeta^2 \right)+
c_\varepsilon r^{\frac{2}{1-\alpha} -2-2\frac{2+\beta}{p-1}}\left(
\rho^{-\frac{2}{1-\alpha}}\ind_{D_{\rho,2\rho}}
+ R^{-\frac{2}{1-\alpha}}\ind_{D_{R,2R}}\right),\\
u^{2}|\nabla\zeta|^2\le & u^{2}
\left[\frac c{\rho^2}\ind_{D_{\rho,2\rho}}\xi^{2\alpha}\eta^2 + \frac c{R^2}\ind_{D_{R,2R}}\xi^2\eta^{2\alpha}\right]\\
\le & \varepsilon \left(r^{-2}u^2\zeta^2 +  r^\beta |u|^{p+1}\zeta^2 \right)+
c_\varepsilon r^{\frac{2}{1-\alpha} -2-2\frac{2+\beta}{p-1}}\left(
\rho^{-\frac{2}{1-\alpha}}\ind_{D_{\rho,2\rho}}
+ R^{-\frac{2}{1-\alpha}}\ind_{D_{R,2R}}\right).
\end{split}\]
Note that, by Hardy inequality~\eqref{hardy}, for all $t>0$,
\[
\int r^{-2}u^2(t)\zeta^2(t)h^2dx\le \frac4{(N+2\lambda -2)^2}\int |\nabla u \zeta|^2(t)h^2dx.
\]
Choose now $\alpha\in(\frac2{p+1},1)$ such that
\[
\frac{2}{1-\alpha} -2-2\frac{2+\beta}{p-1} + 2\lambda + N>0.
\]
Then by a direct calculation
\[
\iint\limits_{D_{R,2R}} R^{-\frac{2}{1-\alpha}}
r^{\frac{2}{1-\alpha} -2-2\frac{2+\beta}{p-1}}
h^2dxdt 
\le R^{N+2\lambda - 2\frac{2+\beta}{p-1}}
\]
which implies that
\[
\iint|\nabla u\zeta|^2 h^2dxdt
+ \iint r^\beta |u|^{p+1}\zeta^2h^2dxdt\le c\rho^{N+2\lambda
- 2\frac{2+\beta}{p-1}} + c R^{N+2\lambda
- 2\frac{2+\beta}{p-1}} + c\int\limits_{B_{2R}\setminus B_\rho}u^2(0)\zeta^2(0)h^2dx,
\]
which completes the proof of the first assertion.

To prove the second assertion, note that by the mean value
inequality for sub-solutions (see Theorem~\ref{harnack} below)
we have that
\[
\sup_{D_{5/2\rho,7/2\rho}} |u|\le C \left({{\displaystyle\dashint}_{\scriptstyle D_{2\rho, 4\rho}}} u^2 h^2 dxdt\right)^\frac12.
\]
Using the Hardy inequality and \eqref{int-KO} with $R=4\rho$, we have
\begin{eqnarray}
\iint_{D_{2\rho, 4\rho}}u^2 h^2 dxdt\le \iint |u\zeta|^2h^2 dxdt\le \rho^2 \iint |u\zeta|^2r^{-2}h^2 dxdt\\
\le c \rho^2 \iint |\nabla(u\zeta)|^2h^2 dxdt\le c \rho^{N+2\lambda + 2
- 2\frac{2+\beta}{p-1}}.
\end{eqnarray}
Hence
\[
\left({{\displaystyle\dashint}_{\scriptstyle D_{2\rho, 4\rho}}} u^2 h^2 dxdt\right)^\frac12\le c\rho^{-\frac{2+\beta}{p-1}}.
\]
\end{proof}

\begin{corollary}\label{subcr-est}
Let $p<1+\frac{2+\beta}{N+2\lambda}$ and let $u$ be a solution
to \eqref{wparabolic} such that $u(t)\to \varkappa \delta_0$ as
$t\to0$ in weak-$*$ topology of Radon measures. Then $u\in
L^2_{loc}\big((0,T);\, H^1_{h^2}\big)$ for all $T>0$ and
$u(x,t)\le \varkappa p^{h^2}_t(x,0)$ for a.a. $(x,t)\in
\mathbb{R}^N\times(0,\infty)$, where $p^{h^2}$ is the
fundamental solution of the linear equation
$(\partial_t-\Delta_{h^2})w=0$. In particular, $u\in
L^{p+1}\big(\mathbb{R}^N\times(0,T),r^\beta h^2dx\, dt\big)$
for all $T>0$.
\end{corollary}

\begin{proof}
The first assertion follows from \eqref{int-KO}, setting $R\to\infty$ and choosing $\rho$ arbitrary small.
The proof of the second assertion literally
follows the argument \cite{kv92}.
Namely, let $u(t)\to \varkappa \delta_0$ as $t\to0$ in weak-$*$
topology of Radon measures. Let $u^{(\tau)}$ be the solution to
the initial value problem
\begin{equation*}
\left\{
\begin{array}{c}
(\partial_t-\Delta_{h^2})u^{(\tau)}=0, \quad t>\tau,
\\
u^{(\tau)}(\tau)=u(\tau).
\end{array}
\right.
\end{equation*}
Then, by the maximum principle, $u^{(\tau)}(x,t)\geq u(x,t)$
for a.a. $(x,t)\in \mathbb{R}^N\times (\tau,\infty)$. So, for
$0<\tau'<\tau$ one has $u^{(\tau')}(\tau)\geq
u(\tau)=u^{(\tau)}(\tau)$. Hence, by the maximum principle,
$u^{(\tau')}(x,t)\geq u^{(\tau)}(x,t)$ a.e. on
$\mathbb{R}^N\times (\tau,\infty)$. So $u^{(\tau)}\uparrow
u^{(0)}$ as $\tau\downarrow0$ and $u^{(0)}(x,t)=\varkappa
p^{h^2}_t(x,0)$ since $u^{(\tau)}(\tau)=u(\tau)\to \varkappa
\delta_0$ as $t\to0$.
\end{proof}

The next lemma reduces the proof of the second assertion of Theorem~\ref{main1} to the critical case
$p=p^*$.

\begin{lemma}\label{down}
Let $u$ be a sub-solution to \eqref{wparabolic} satisfying
\eqref{null}. Then, for $1<q\leq p$, the function
\[\left(\frac{p-1}{q-1}\right)^{\frac1{p-1}}|u|^{\frac{p-1}{q-1}}\]
is a sub-solution to the equation $\partial_t w -
h^{-2}\nabla\cdot(h^2\nabla w) + r^\beta w|w|^{q-1} =0$.
\end{lemma}

\begin{proof}
Denote $\vark=\frac{p-1}{q-1}\geq1$, $T_\delta$ the Steklov
average and $u_\delta :=T_\delta u$. For $\varepsilon>0$ and
$\zeta\in C^{2,1}_c\Big(\mathbb{R}^N\times(0,\infty)\Big)$,
$\zeta\geq0$, choose the following test function for
\eqref{wparabolic}:
\[T_\delta\Big(\zeta u_\delta (u_\delta ^2 +\varepsilon)^{\frac{\vark-2}2}\Big).\]
Note that this is a legitimate test function since $u$ is
locally bounded. Then the following inequality holds:
\begin{equation}\label{power}
\begin{split}
\iint \partial_tu_\delta  \zeta u_\delta (u_\delta ^2 +\varepsilon)^{\frac{\vark-2}2} h^2dxdt
&+ \iint \nabla u_\delta  \nabla (\zeta u_\delta (u_\delta ^2 +\varepsilon)^{\frac{\vark-2}2})h^2dxdt\\
\leq & - \iint \zeta u_\delta (u_\delta ^2
+\varepsilon)^{\frac{\vark-2}2}r^\beta T_\delta(|u|^{p-1}u)h^2dxdt.
\end{split}
\end{equation}
Denote $V_\varepsilon(u):=\frac1\vark
\Big((u^2+\varepsilon)^{\frac\vark 2} -
\varepsilon^{\frac\vark 2}\Big)$. Then
\[
\partial_tu_\delta  \zeta u_\delta (u_\delta ^2 +\varepsilon)^{\frac{\vark-2}2} = \zeta \partial_tV_\varepsilon(u_\delta )
\]
and
\[
\nabla u_\delta  \nabla (\zeta u_\delta (u_\delta ^2 +\varepsilon)^{\frac{\vark-2}2}) =
\nabla V_\varepsilon(u_\delta )\nabla\zeta + \zeta \Big((\vark-1)u_\delta ^2 + \varepsilon\Big)(u_\delta ^2 +\varepsilon)^{\frac{\vark-4}2}|\nabla u_\delta |^2.
\]
Hence it follows from \eqref{power} that
\[
-\iint V_\varepsilon(u_\delta )\partial_t \zeta h^2dxdt
+ \iint \nabla V_\varepsilon(u_\delta ) \nabla\zeta h^2dxdt
\leq  - \iint \zeta u_\delta (u_\delta ^2
+\varepsilon)^{\frac{\vark-2}2}r^\beta T_\delta(|u|^{p-1}u)h^2dxdt.
\]
Passing to the limit as $\varepsilon\to0$ and then as
$\delta\to0$ we obtain
\[
-\iint |u|^\vark\partial_t \zeta h^2dxdt
+ \iint \nabla |u|^\vark \nabla\zeta h^2dxdt
+ \vark \iint \zeta |u|^{p + \vark -1}h^2dxdt\leq 0.
\]
Hence the assertion follows.
\end{proof}

\begin{remark}
Lemma~\ref{down} is a parabolic version of \cite[Proposition~1.1]{KLS09}.
\end{remark}

The following theorem establishes the removability of singularity at $(0,0)$ for the critical case.

\begin{theorem}
\label{thm6.1} Let $p=1+\frac{2+\b}{N+2\l}$. Let $0\le u\in
L^\infty_{loc}(\R^N\times \R^1_+\setminus{(0,0)})$ be such that
\begin{equation}
\label{e6.1}
\iint_ {\R^N\times \R^1_+}\left(r^\b u^p\zeta-u\zeta_t-u\Delta_{h^2} \zeta\right)h^2dx\,dt \le 0, \quad \zeta \in C^{2,1}_c.
\end{equation}
Then $u=0$.
\end{theorem}
\begin{proof}
Let $\xi\in \C^1(\R^1_+)$ be such that
\begin{equation}
\label{e6.2}
\ind_{[2,\infty)}\le \xi \le \ind_{[1,\infty)}, \quad |\xi\pr|,|\xi^{''}|\le c \xi^\frac1p.
\end{equation}
Let $0<\rho \ll R<\infty$ and define
\begin{equation}
\label{e6.3}
\xi_\rho(x,t)=\xi\left(\frac{t+|x|^2}{\rho^2}\right),\quad \eta_R(x,t)=1-\xi\left(\frac{t+|x|^2}{R^2}\right).
\end{equation}
We take $\zeta=\xi_\rho\eta_R$ as a test function  in
\eqref{e6.1}. It is easy to see that
\[
{\rm supp}\,\zeta=\{(x,t)\,:\, \rho^2\le t+|x|^2\le 2R^2\}.
\]
Using \eqref{e6.2} one verifies directly  that
\[
|\partial_t \zeta|+|\Delta_{h^2}\zeta|\le c\frac1{\rho^2}\xi_\rho^\frac1p\eta_R\ind_{\{\rho^2\le t+|x|^2\le 2r^2\}}+
c\frac1{R^2}\xi_\rho\eta_R^\frac1p\ind_{\{R^2\le t+|x|^2\le 2R^2\}}.
\]
Thus we have
\begin{equation}\label{e6.4}
\begin{split}
I:=&\iint_ {\R^N\times \R^1_+}r^\b u^p\zeta h^2dx\,dt \le\iint_ {\R^N\times \R^1_+}u(|\partial_t \zeta|+|\Delta_{h^2}\zeta| )h^2dx\,dt\\
\le &c \rho^{-2}\iint_ {\R^N\times \R^1_+}u\xi_\rho^\frac1p\eta_R\ind_{\{\rho^2\le t+|x|^2\le 2\rho^2\}}h^2dx\,dt\\
& + cR^{-2} \iint_ {\R^N\times \R^1_+}u\xi_\rho\eta_R^\frac1p\ind_{\{R^2\le t+|x|^2\le 2R^2\}}h^2dx\,dt\\
&:= I_1+I_2.
\end{split}
\end{equation}
By the Young inequality
\begin{equation*}
I_1\le c \rho^{-2}\iint_ {\R^N\times \R^1_+}u(\xi_\rho\eta_R)^\frac1p\ind_{\{\rho^2\le t+|x|^2\le 2\rho^2\}}h^2dx\,dt
\le \frac14 I+ c \rho^{N+2-\frac{2p}{p-1}+2\l-\frac{\b}{p-1}}.
\end{equation*}
Similarly,
\begin{equation*}
I_2\le c R^{-2}\iint_ {\R^N\times \R^1_+}u(\xi_\rho\eta_R)^\frac1p\ind_{\{R^2\le t+|x|^2\le 2R^2\}}dx\,dt
\le \frac14 I+ c R^{N+2-\frac{2p}{p-1}+2\l-\frac{\b}{p-1}}.
\end{equation*}
Hence for every $\rho>0$ and $R>2\rho$ we obtain
\[
I\le c(\rho^{N+2-\frac{2p}{p-1}+2\l-\frac{\b}{p-1}}+R^{N+2-\frac{2p}{p-1}+2\l-\frac{\b}{p-1}})=c\quad \text{as}\ \ N+2-\frac{2p}{p-1}+2\l-\frac{\b}{p-1}=0.
\]
Passing to the limits $\rho\to 0$ and $R\to\infty$ we conclude
that
\begin{equation}
\label{e6.6}
\iint_ {\R^N\times \R^1_+}r^\b u^p\zeta h^2dx\,dt<\infty.
\end{equation}
Now we return to \eqref{e6.4}. Estimating $I_1$ and $I_2$ by
means of the Young inequality and using \eqref{e6.6} we have
\begin{eqnarray}
\nonumber
I_1&\le &c \rho^{-2}\iint_ {\R^N\times \R^1_+}u\xi_\rho^\frac1p\ind_{\{\rho^2\le t+|x|^2\le 2\rho^2\}}h^2dx\,dt \\
&\le & c\left(\iint_ {\R^N\times \R^1_+}\ind_{\{\rho^2\le t+|x|^2\le 2\rho^2\}}r^\b u^p\zeta hh^2dx\,dt\right)^\frac1p \to 0 \quad \text{as}\ \rho\to 0 .
\end{eqnarray}
Similarly we see that $I_2\to 0$ as $R\to \infty$. Hence we
conclude from \eqref{e6.4} that $I=0$ which implies that $u=0$.
\end{proof}
Now assertion (c)  of Theorem~\ref{main1} follows from Theorem~\ref{thm6.1}, Lemma~\ref{down} and the corresponding parabolic version
of the Kato inequality (see, e.g. \cite[Chap.~6]{Ver}).

\section{Linear equation with a generator of an ultra-contractive Feller semigroup}

In this section we study an abstract inhomogeneous evolution equation with measures as initial data.

In this section $\Omega\subset \mathbb{R}^N$ is a domain and $\g$ is a positive Radon measure on $\O$ and we
denote $L^p:=L^p(\O,d\g)$. Let $T>0$ and
$Q=\Omega\times [0,T]$. We also denote $L^p(Q)=L^p(Q,d\g\,dt)$ and naturally identify $L^p(Q)=L^p([0,T];\, L^p)$.

In the sequel we also use the notation $C_0(\O), C_b(\O)$ for the spaces of continuous function vanishing at infinity and at the boundary
of $\O$ and bounded continuous function, respectively. $\mathcal{M}(\Omega)$ stands for finite signed Radon measures on $\O$.

Let $(\mathcal{E,F})$ be a closed symmetric Dirichlet form  on
$L^2$, $-\gen$ the associated self-adjoint operator in $L^2$,
and $S=(S(t))_{t\ge 0}$ the associated symmetric Markov semigroup on $L^2$, i.e. $\|S(t)f\|_\infty\le \|f\|_\infty$ for any $t\ge 0$),
$S(t)=e^{\gen t}$.  The domain $\mathcal{F}$ of the form $\mathcal{E}$ is a real Hilbert space with the norm
$\|f\|_{\mathcal F}=\left(\mathcal{E}(f)\right)^\frac12$.
We refer the reader to \cite{Da,FOT} for the definition and properties.

The action of the semigroup $S$ on the measure $\mu$ is defined in a standard way by
the following identity
\begin{equation}
\label{Smeas}
\int_\O(S(t)\mu)\phi d\g=\int_\O (S(t)\phi)d\mu, \quad \phi \in C_0(\O).
\end{equation}

We start with the following simple statement.
\begin{proposition}\label{SG}
Let $\psi:(0,\infty)\to(0,\infty)$ be a non-increasing function.
Assume that
\begin{equation}\label{ultra}
    \|S(t)\|_{L^1\to L^\infty}\le \psi(t),~t>0\,
\end{equation}
and
\begin{equation}\label{feller}
S(t)C_0(\Omega)\subset C_b(\Omega) \mbox{ and }S(t)\ind\in C_b(\Omega).
\end{equation}
Then
$S(t)$, $t>0$,  is a bounded
      operator $S(t): \mathcal{M}(\Omega)\to C_b(\Omega)\cap L^1\cap
      \mathcal{F}$
      and
      \begin{equation}
      \label{MtoL}
      \|S(t)\|_{\mathcal{M}\to L^{q}}\le \psi^{\frac1{q'}}(t), \quad 1\le q\le \infty.
      \end{equation}

Moreover, for every $t>0$, $S(t)$ is an integral operator with
a positive bounded symmetric kernel $p_t(x,y)$ which is
continuous in each of the variables $x,y,t$ and 
\begin{equation}\label{kern-est}
 p_t(x,y)\le \psi(t) \mbox{ and } \int\limits_\Omega p_t(x,y)\gamma(dy)\le1.
\end{equation}
For every $t>0$, the operator $S(t)$ maps weak-$*$-convergent sequences in
$\mathcal{M}(\Omega)$ into strongly convergent sequences in
$C_b(\Omega)\cap L^1\cap
      \mathcal{F}$.
\end{proposition}

\begin{proof}
By the Riesz-Thorin interpolation theorem it follows from
\eqref{ultra} that  $\|S(t)\|_{L^p\to L^q}\le \psi^{\frac1p
-\frac1q}(t)$, $1\le p\le q\le\infty$. Since $C_0(\Omega)\cap
L^p$ is dense in $L^p$, by \eqref{feller}, $S(t):L^p\to
C_b(\Omega)$. Hence, $\|S(t)\|_{L^p\to C_b}\le
\psi^{\frac1p}(t)$, $t>0$. By duality $S(t):C_b(\Omega)^*\to
L^q$, $1\le q\le \infty$. In particular,
$S(t):\mathcal{M}(\Omega)\to L^q$ and $\|S(t)\|_{\mathcal{M}\to
L^{q}}\le \psi^{\frac1{q'}}(t)$, $1\le q\le \infty$. By the
simple factorization $S(t)=S(t/2)S(t/2):\mathcal{M}(\Omega)\to
L^q \to C_b(\O)$, and $\|S(t)\|_{\mathcal{M}\to C_b}\le
\psi(t)$, $t>0$.

 Similarly,
$S(t)=S(t/2)S(t/2):\mathcal{M}(\Omega)\to L^2 \to \mathcal{F}$.

The second assertion follows from the first one taking $p_t(x,y)=(S(t)\d_y)(x)$.

To prove the last assertion we first show that $S$
 is a strong  Feller semigroup, that is,   for $t>0$,  $S(t)$ maps bounded Borel
measurable functions into continuous ones. To this end it
suffices to verify that $x\mapsto p_t(x,\cdot)$ is a continuous
function from $\Omega$ to $L^1$ for all $t>0$. If
$\gamma(\Omega)<\infty$ this immediately follows from the fact
that $p_t(x,y)$ is continuous in $x$ for all $t>0$ and
$y\in\Omega$ and the bound $0\le p_t(x,y)\le \psi(t)$. In case
$\gamma(\Omega)=\infty$ to verify the assumptions of the Vitali theorem it suffices
to show that, for every $x_n\to x$ in $\Omega$ as $n\to\infty$
and every $\varepsilon>0$, there exists a compact
$K_\varepsilon\subset \Omega$ and $N_\varepsilon\in \mathbb{N}$
such that
\[
\int\limits_{\Omega\setminus K_\varepsilon}p_t(x_n,y)\gamma(dy)<\varepsilon \mbox{ for all }n>N_\varepsilon.
\]
Given $x_n\to x$ in $\Omega$ as $n\to\infty$ and
$\varepsilon>0$, let $K_\varepsilon\subset \Omega$ be such that
\[
\int\limits_{\Omega\setminus K_\varepsilon}p_t(x,y)\gamma(dy)<\frac\varepsilon2.
\]
Note that $\ind_{K_\varepsilon}\in L^1$ so
$S_t\ind_{K_\varepsilon}\in C_b(\Omega)$. Since $S_t\ind \in
C_b(\Omega)$, we conclude that
$S_t\ind_{\Omega\setminus
K_\varepsilon}=S_t\ind-S_t\ind_{K_\varepsilon}\in C_b(\Omega)$.
In particular,
\[
\left|\int\limits_{\Omega\setminus K_\varepsilon}p_t(x_n,y)\gamma(dy)
- \int\limits_{\Omega\setminus K_\varepsilon}p_t(x,y)\gamma(dy)\right|\to 0 \mbox{ as }n\to\infty.
\]
Now choose $N_\varepsilon$ such that the above variable is less
then $\frac\varepsilon2$ for $n>N_\varepsilon$. 
Thus $S$ is strongly Feller.

Now let $\mu_n\to \mu$ as $n\to\infty$ in the sense of
weak-$*$ convergence in $\mathcal{M}(\Omega)$. Then, for every
Borel measurable $E$,
\[
\int\limits_E \big(S(t)\mu_n\big) d\gamma = \int\limits_\Omega (S(t)\ind_E)d\mu_n
\to \int\limits_\Omega (S(t)\ind_E)d\mu = \int\limits_E \big(S(t)\mu\big) d\gamma
\mbox{ as }n\to\infty,
\]
since $S(t)\ind_E\in C_b(\Omega)$. Hence $S(t)\mu_n\to S(t)\mu$
as $n\to\infty$ weakly in $L^1$. Since
$\big(S(t)\mu_n\big)(x)=\int p_t(x,y)\mu_n(dy)$ and $p_t$ is
bounded and continuous in $y$, we conclude that $S(t)\mu_n\to
S(t)\mu$ as $n\to\infty$, pointwise in $\Omega$. Hence
$S(t)\mu_n\to S(t)\mu$ as $n\to\infty$ strongly in $L^1$. The
strong convergence in $\mathcal{F}$ and in $C_b(\Omega)$
follows from the factorization argument.
\end{proof}

Let us introduce
the convolution operator $\convl$ on $L^p(Q)$ by
$(\convl f)(x,t)= \int_0^t \left(S(t-s)f\right)(x, s)ds$.
In the next two propositions we collect the required properties of $\convl$.

\begin{proposition}
\label{SG-new}
Let condition \eqref{ultra} hold.
The following assertions hold
\begin{enumerate}
\item[1)] $\convl$ is
 a completely continuous operator on
      $L^1(Q)$;
  \item[2)] $\convl$ is a bounded operator ${L}^2(Q)\to
      L^2\left((0,T);\, \mathcal{F}\right)$ and
      ${L}^2(Q)\to L^2\left((0,T);\,
      D(\mathcal{L}) \right)$.
 \end{enumerate}
\end{proposition}
\begin{proof}
Note that $\convl$ is an integral operator on $L^1(Q)$ and
\[
(\convl f)(x,t)=\iint\limits_{Q}k(x,t;y,s)f(y,s)\g (dy)\,ds
\]
with $k(x,t;y,s)=\ind_{(0,\infty)}(t-s)p_{t-s}(x,y)$. Since
\[
\iint\limits_{Q}k(x,t;y,s)\g(dx)\,dt
=\int\limits_0^T\ind_{(0,\infty)}(t-s)\int\limits_\Omega p_{t-s}(x,y)\g(dx)\, dt
\le \int\limits_s^T dt \le T
\]
for a.a. $(y,s)\in Q$, it follows by the Dunford--Pettis lemma
(see e.g. \cite[Lemma III.11]{du77}) that $\mathcal{T}$ is
completely continuous on $L^1(Q)$.

To prove the next assertion, observe that
\[
\|(\convl f)(t)\|_{\mathcal{F}}\leq c\int_0^t\frac1{\sqrt{t-s}}\|f\|_{L^2}(s)ds \quad \mbox{ and }\quad
\|(\mathcal{L}\convl f)(t)\|_{L^2}\leq c\int_0^t\frac1{{t-s}}\|f\|_{L^2}(s)ds.
\]
Since the integral operators with the kernels 
$K_0(t,s)=\frac1{\sqrt{t-s}}$ and $K_1(t,s)=\frac1{{t-s}}$ are
bounded on $L^2(0,T)$, the second assertion follows.
\end{proof}

\begin{proposition} Let the conditions of Proposition~\ref{SG} be fulfilled.
In addition assume that
\begin{equation}
\label{w-cont}
S(t)\phi(x)\to \phi(x)\mbox{ as }t\to0
\mbox{ for all } x\in\Omega,\phi\in C_0(\Omega).
\end{equation}
Then $\convl$ can be uniquely extended to a bounded operator from $\mathcal{M}(Q)$ to
      $L^\infty\left((0,T);\,L^1\right)$. Moreover,
      $\convl L^1(Q)\subset C\left([0,T];\,L^1\right).$
\end{proposition}
\begin{proof}
First, observe that, $S(t)\mu$ is strongly continuous in
$L^1\subset\mathcal{M}(\Omega)$ for all $t>0$, for every $\mu\in
\mathcal{M}(\Omega)$.   Moreover, it follows from \eqref{w-cont} that
$S(t)\mu$ is $w$-$*$ continuous
continuous at $t=0$.

Now let $m\in\mathcal{M}(Q)$ and $m=\mu_t\otimes\nu$ be its
disintegration into $\nu\in \mathcal{M}([0,T])$ and a
function $t\mapsto \mu_t\in \mathcal{M}(\Omega)$
such that $t\mapsto \mu_t(F)$ is $\nu$-measurable for all Borel sets
$F$ (see \cite[Theorem~2.28]{afp06}). So $t\mapsto \mu_t$ is a weakly $\nu$-measurable
function from $[0,T]$ to $\mathcal{M}(\Omega)$. Hence the function $s\mapsto S(t-s)\mu_s$ is
also a weakly $\nu$-measurable
function from $[0,t]$ to $\mathcal{M}(\Omega)$. Since $S(s)\mathcal{M}(\Omega)\subset L^1$
for $s>0$, we conclude that $s\mapsto S(t-s)\mu_s$ is separably valued, hence it is (strongly)
$\nu$-measurable by the Pettis measurability theorem (see \cite[Theorem~2.2]{du77}).
So we define the extension of $\convl$ on  $\mathcal{M}(Q)$ by
\begin{equation}
\label{Tau-meas}
(\convl m)(t):= \int\limits_{[0,t]}S(t-s)\mu_s\nu(ds),
\end{equation}
where the right hand side is a Bochner integral.

Moreover, $\convl : \mathcal{M}(Q) \to L^\infty((0,T);\,\mathcal{M}(\Omega))$ is bounded. Indeed,
\[
\left\|
(\convl m)(t)
\right\|_{\mathcal{M}(\Omega)}
\leq \int_0^t \left\|S(t-s)\mu_s\right\|_{\mathcal{M}(\Omega)}|\nu|(ds)
\leq \int_0^t \left\|\mu_s\right\|_{\mathcal{M}(\Omega)}|\nu|(ds)
= \|m\|_{\mathcal{M}(Q)}.
\]
Hence the extension is unique.

Now, let $\nu=\nu_c + \sum c_k\delta_{t_k}$ be the
decomposition of $\nu$ into the continuous and the atomic
parts. Then $\int_0^tS(t-s)\mu_s\nu_c(ds) \in
L^1$ since $S(t-s)\mu_s \in L^1$
for all $s\in[0,t)$, and
\[
\int_0^tS(t-s)\mu_s\sum c_k\delta_{t_k}(ds) = \sum\limits_{t_k\leq t} c_kS(t-t_k)\mu_{t_k}.
\]
The latter belongs to $L^1$ for all $t\ne t_k$,
$k=1,2,\ldots$. So $\convl m (t)\in L^1$ for a.a
$t\in[0,T]$.

Finally, we show that if $\nu=\nu_c$ then $\convl m\in
C\big([0,T];\, L^1\big)$, which will prove the last assertion. Indeed,
\[
\convl m(t+h) - \convl m(t)
= \int\limits_t^{t+\delta}S(t + \delta -s)\mu_s\nu(ds)
+\int\limits_0^t\Big[S(t+\delta-s) - S(t-s)\Big]\mu_s\nu(ds).
\]
Then
\[
\Big\|\int\limits_t^{t+\delta}S(t + \delta -s)\mu_s\nu(ds)\Big\|_{L^1}\le
\int\limits_t^{t+\delta} \|\mu_s\|_{\mathcal{M}(\Omega)}|\nu|(ds)=|m|\big(\Omega\times(t,t+\delta)\big)\to0
\mbox{ as }\delta\to0.
\]
Further,
\[\Big\|\big[S(t+\delta-s) -
S(t-s)\big]\mu_s\Big\|_{L^1}\to 0 \mbox{ as }\delta\to0
\mbox{ for all }s\in[0,t).\] Moreover, $\Big\|\big[S(t+\delta-s) -
S(t-s)\big]\mu_s\Big\|_{L^1}\le
2\|\mu_s\|_{\mathcal{M}(\Omega)}$. Thus
\[
\Big\| \int\limits_0^t\Big[S(t+\delta-s) - S(t-s)\Big]\mu_s\nu(ds)\Big\|_{L^1}\to0
\mbox{ as }\delta\to0.
\]
\end{proof}

\begin{remark}
For further use we observe that, for $\eta\in C\big([0,T];\, L^1\cap C_b\big)$,
\[
(\convl^*\eta)(t)=\int\limits_t^TS(s-t)\eta(s)ds,
\]
where the right hand side is a Bochner integral. Note that $\convl^*$ is a bounded
operator on $C\big([0,T];\, L^1\cap C_b\big)$, by the argument similar to the one in the proof of
the preceding proposition.
\end{remark}

\begin{definition}\label{wsolution}
Let $m\in \mathcal{M}(Q)$ and $\mu\in\mathcal{M}(\Omega)$. We
say that $u\in L^2_{loc}\Big((0,\infty);\, \mathcal{F}\Big)\cap
L^1(Q)$ is a solution (sub-solution) to the problem
\begin{equation}
\label{m-Cauchy}
(\partial_t -
\gen)u = m, \quad u(0)= \mu
\end{equation}
if the following integral identity (inequality)
holds
\begin{eqnarray}
\nonumber
    &&\int\limits_\Omega u(t_1)\zeta(t_1)\g(dx) -\int\limits_{t_0}^{t_1}\int\limits_\Omega u \partial_t\zeta\,\g(dx)\,dt + \int\limits_{t_0}^{t_1}\mathcal{E}\big(u(t),\zeta(t)\big)dt\\
    \label{wequ}
&& =\, (\leq)\,  \int\limits_{t_0}^{t_1}\int\limits_\Omega\zeta\, m(dxdt)
+ \int\limits_\Omega u(t_0)\zeta(t_0)\g(dx)\,
\end{eqnarray}
and
\begin{equation}\label{winit}
\lim\limits_{t\to0}\,(\limsup\limits_{t\to0})\,\int\limits_\Omega
u(t)\zeta(t)\g(dx) =\,(\leq)
\, \int\limits_\Omega \zeta(0)d\mu
\end{equation}
for all $t_1>t_0>0$ and $\zeta\in W$, $\zeta\geq0$, where
\[
W:=\left\{\zeta\in C_b(Q)\cap
L^2_{loc}\Big((0,\infty);\,\mathcal{F}\Big)\cap
W^{1,\infty}_{loc}\Big((0,\infty);\, L^\infty\Big)\right\}.
\]
\end{definition}

The next lemma provides the representation of the solution to \eqref{m-Cauchy}.

\begin{lemma}\label{uniq-non-hom}
Let $u$ be a solution (sub-solution) to \eqref{m-Cauchy}.
Then $u =S\mu + \convl m$ \ ($u \le S\mu + \convl m$).

\end{lemma}
\begin{proof}
We prove the assertion for solutions, the proof for sub-solutions being completely the same.
Let $\eta\in C\big([0,T];\, L^1\cap C_b\big)$, $\lambda>0$. Denote
\[\eta_\lambda(t):=(\mathrm{I} - \lambda\gen)^{-1}\eta(t), t\in[0,T] \mbox{ and }
\zeta_\lambda:=\convl^*\eta_\lambda.
\]
Then $\eta_\lambda, \zeta_\lambda \in  C\big([0,T];\, L^1\cap C_b\big)$. Since
$\gen
\eta_\lambda(t)=\frac1\lambda\eta(t) - \frac1\lambda\eta_\lambda(t)$ for all $t\in[0,T]$, we conclude that
$\gen\eta_\lambda(\cdot),\, \gen\zeta_\lambda(\cdot) \in  C\big([0,T];\, L^1\cap C_b\big)$. Hence $\partial_t
\zeta_\lambda = -\gen \zeta_\lambda - \eta_\lambda.$ In
particular, $\zeta_\lambda\in W$.

Testing \eqref{wequ} by $\zeta_\lambda$ and noticing that $\xi_\lambda(T)=0$ we obtain
\[
-\int_{t_0}^T \int_\O u \partial_t \zeta_\lambda d\g\, dt
+ \int_{t_0}^T \mathcal{E}(u,\zeta_\lambda)dt = \int_{t_0}^T \int_\O \zeta_\lambda dm
+ \int_\O u(t_0) \zeta_\lambda(t_0)d\gamma.
\]
Note that
\[
\mathcal{E}(u,\zeta_\lambda) = - \int_\O u \gen \zeta_\lambda d\gamma = \int_\O u (\partial_t \zeta_\lambda + \eta_\lambda)d\gamma.
\]
Hence, passing to the limit as $t_0\to0$ we obtain that
\[
\int\limits_0^T\int\limits_\Omega u\eta_\lambda d\g\, dt = \int_{0}^T \int_\O \zeta_\lambda dm
+ \int_\O \zeta_\lambda(0)d\mu = \int\limits_0^T \int\limits_\Omega (\convl m) \eta_\lambda d\g\, dt
+ \int_0^T\int\limits_\Omega (S\mu) \eta_\lambda d\g\, dt,
\]
where the last equality follows from  \eqref{Smeas} and the definition of $\convl$. 
Finally, observe that $\eta_\lambda\to\eta$ as $\lambda\to0$
pointwise, so passing to the limit in $\lambda$, we have
\[
\int\limits_0^T\int\limits_\Omega u\eta\, d\g\, dt =\int\limits_0^T\int\limits_\Omega (\convl m + S\mu)\eta\, d\g dt.
\]
Hence the assertion follows.
\end{proof}
The next proposition gives a version of a maximum principle. It is an extension of \cite[Lemma~3]{bf83}.

\begin{proposition}\label{positive}
Let $f\in {L}^1(Q)$,
$\mu\in\mathcal{M}(\Omega)$, and $u$
 be a solution to \eqref{m-Cauchy}. 
Then, for $t\in(0,T)$,
\[
\begin{split}
\int\limits_\Omega u^+(t) d\g &
\leq
\int\limits_0^t \int\limits_{\Omega}f\ind_{\{u>0\}}d\g ds
+ \int\limits_{\Omega}d\mu^+,
\\
\int\limits_\Omega |u(t)| d\g &\leq \int\limits_0^t \int\limits_{\Omega}f\sgn(u)  d\g ds
+ \int\limits_{\Omega}d|\mu|.
\end{split}
\]
\end{proposition}

\begin{proof}
Note that $u=S\mu + \convl f$, by Lemma \ref{uniq-non-hom}. 
It suffices to prove the inequalities for
$f\in L^1(Q)\cap L^2(Q)$
since $\convl$ is a bounded operator on $L^1(Q)$.
By Proposition~\ref{SG-new} $u\in
L^2_{loc}\Big((0,T);\,\mathcal{F}\Big)$,
$\partial_t u, \, \gen u\in
L^2_{loc}\Big((0,T);\,{L}^2\Big)$ and
\begin{equation}\label{positive-appr}
    (\partial_t - \gen)u = f.
\end{equation}

Now we prove the first estimate. Denote $v_k(s):=(ks)^+\wedge
1$, $k=1,2,\ldots$ Then $v_k$ is Lipschiz, non-decreasing,
$v_k(0)=0$ and $v_k\to \ind_{\{s>0\}}$ as $k\to\infty$.
Hence $v_k(u)\in L^2_{loc}\Big((0,T);\,\mathcal{F}\Big)$ (cf.~\cite[Theorem~1.4.1]{FOT}).

We claim that
$\mathcal{E}\left(v_k(u),u\right)\geq0$ . 
Indeed, recall that, for all $u,v\in \mathcal{F}$ one has
$\mathcal{E}(u,v)=\lim\limits_{\lambda\to\infty}\mathcal{E}^\lambda(u,v)$,
where
\[
\mathcal{E}^\lambda(u,v)=\mathcal{E}\left(u,\lambda(\lambda-\gen)^{-1}v\right)
\]
is the approximation of $\mathcal{E}$.
By \cite[(1.4.8)]{FOT}, there exist positive measures $\mu_\lambda\in
\mathcal{M}(\Omega)$ and $\sigma_\lambda\in
\mathcal{M}(\Omega\times \Omega)$ such that
\[
\mathcal{E}^\lambda(u)=\int_\Omega u^2\mu_\lambda(dx) + \iint_{\Omega\times \Omega}\left(u(x)-u(y)\right)^2\sigma_\lambda(dx,dy).
\]
Then it is straightforward that
$\mathcal{E}^\lambda(\rho(u),u)\geq 0$ for all $\lambda>0$ and
all Lipschiz monotone $\rho$ such that $\rho(0)=0$.
Hence passing to the limit as
$\lambda\to\infty$, we conclude that $\mathcal{E}(v_k(u),u)\geq 0$.

Now multiply \eqref{positive-appr} by $v_k(u)$ in ${L}^2$
to obtain that
\[
\int\limits_\Omega v_k(u(s)) \partial_t u(s)\,d\g \leq \int\limits_\Omega v_k(u) f\,d\g.
\]
Integrating the latter in $s$ over the interval $(\tau,t)$ we obtain
\[
\int\limits_\Omega V_k\big(u(t)\big)d\g \leq \int_\tau^t\int\limits_\Omega v_k(u) f\,d\g ds
+ \int\limits_\Omega V_k\big(u(\tau)\big)d\g,
\]
where $V_k(s)$ is the primitive of $v_k(s)$, $V_k(s)\uparrow
s^+$ as $k\uparrow\infty$. So, for $0<\tau<t$, it follows that
\[
\int\limits_\Omega u^+(t)d\g \leq \int\limits_\tau^t\int\limits_\Omega f\ind_{\{u>0\}}d\g ds
+ \int\limits_\Omega u^+(\tau)d\g.
\]
It remains to pass to the limit $\tau\to 0$. By
Lemma~\ref{uniq-non-hom} using positivity of $S$ and $\convl$, we
have that $u^+(\tau)=(S(\tau)\mu + (\convl f)(\tau))^+ \le
S(\tau)\mu^+ + (\convl f^+)(\tau)$ and
\[
\int\limits_\Omega (\convl f^+)(\tau) d\g = \int\limits_0^\tau \int\limits_\Omega S(\tau-s)f^+(s)d\g ds
\le \int\limits_0^\tau \|f(s)\|_{L^1}ds\to0 \mbox{ as }\tau\to0.
\]
So, as $\tau\to0$ we arrive at the first assertion.

To prove the second assertion, note that $v=(-u)$ is the
solution to the problem $(\partial_t - \gen)v= -f$,
$v(0)=-\mu$. Hence
\[
\int\limits_\Omega u^-(t)d\g \leq -\int\limits_0^t\int\limits_\Omega f\ind_{\{u<0\}} d\g ds
+ \int\limits_\Omega d\mu^-.
\]
\end{proof}

We conclude this section by recalling two results on the parabolic equation with a weighted Laplacian.
\paragraph{Linear equation for a weighted Laplacian.}
Here
we consider a special of the measure $d\g =h^2 dx$ and the operator $\gen u=- h^{-2}\div(h^2\nabla u)$, where as before $h(x)=|x|^\l$ with $\l>\frac{2-N}2$.
Namely, we state the 
Mean-value inequality
and the heat kernel estimates for
the linear
equation
\begin{equation}\label{linear}
\partial_t u - h^{-2}\div(h^2\nabla u)=0.
\end{equation}

\begin{theorem}[
Mean-value inequality]\label{harnack}
There exists a constant $C>0$ such that, for all $(x,t)\in
\mathbb{R}^{N+1}$, $r>0$, $q>0$ and a weak positive
(sub-)solution $u$ to \eqref{linear} in the cylinder
$Q^{(x,t)}_{2r}:= B_{2r}(x)\times(t-4r^2,t+4r^2)$, the
following inequality holds: for $Q^-:=B_{r/2}(x)\times (t-2r^2,t)$
and $Q^+:=B_r(x)\times (t+3r^2,t+4r^2)$,
\[
\sup\limits_{Q^-}u\le C\left(\dashint_{ _{\scriptstyle Q^+}}u^q\right)^\frac1q,
\]
where the average integral in the right hand side is by measure
$h^2dxdt$.
\end{theorem}

\begin{theorem}\label{hk}
Let $k$ be the fundamental solution $k$ to the equation \eqref{linear}.
Then
for all $\delta>0$ there
exists $c_\delta>0$ such that for all $x,y\in \R^N$ and $t>0$ the following estimate holds:
\begin{equation}\label{hkb}
k(t,x,y)\le c_\delta t^{-\frac {N+2\lambda}2}e^{-\frac{|x-y|^2}{4(1+\delta)t}}\left( \frac{|x|}{\sqrt t}
+ 1\right)^{-\lambda} \left( \frac{|y|}{\sqrt t}
+ 1\right)^{-\lambda}.
\end{equation}

\end{theorem}

The detailed exposition of these and related results can be found in \cite{gs-c05,MT}.

\section{Source solutions }

Here we use the same notation as in the previous section.
In this section we construct solutions to an abstract semilinear equation with measures as initial data.
We closely follow ideas from \cite[Chapter~6]{Ver}.

Consider the solution of the non-linear
equation
\begin{equation}
\label{nonlinear1}
\left(\partial_t - \gen\right)u(x,t) +
g\left(x,u(x,t)\right)=0,\quad  u(0)=\mu\in
\mathcal{M}(\Omega),
\end{equation}
where $\gen$ is as in the previous
section and $g:\Omega\times \mathbb{R}\to \mathbb{R}$ is
measurable in $x$ for all $r\in\mathbb{R}$, continuous in $r$
for a.a. $x\in\Omega$ (the Caratheodory conditions),
non-decreasing in $r$ and vanishing at $r=0$ for a.a
$x\in\Omega$. We denote $G:u\mapsto g\left(x,u(x)\right)$ the
correspondent monotone homogeneous Nemytskii operator. So a weak
solution to the problem \eqref{nonlinear1}, e.g.
$(\partial_t - \gen + G)u=0$,
$u(0)=\mu$, is $u\in L^1(Q)\cap L^2_{loc}((0,T);\, \cal{F})$ such that $Gu\in
L^1(Q)$ and $(\partial_t - \gen)u = - Gu$,
$u(0)=\mu$ in the sense of Definition~\ref{wsolution}. In
particular,
\begin{equation}
\label{repres}
u\ =S\mu - \convl Gu.
\end{equation}

\begin{proposition}\label{monotone}
Let $\mu_1,  \mu_2\in \mathcal{M}(\Omega)$, $\mu_1\le \mu_2$,  $g_1(x,r)\ge
g_2(x,r)$ for all $r\in\mathbb{R}$ and a.a. $x$,  $G_1,G_2$
be the corresponding Nemytskii operators and $u_j\in L^1(Q)$
be solutions to the problems $(\partial_t - \gen + G_j)u_j=0$,
$u_j(0)=\mu_j$, $j=1,2$. Then $u_1\leq u_2$ pointwise for a.a. $(x,t)\in Q$.
\end{proposition}

\begin{proof}
Let $w=u_1-u_2$. Then $w$ satisfies $(\partial_t - \gen)w =
-(G_1u_1-G_2u_2),$ $w(0)= -(\mu_2-\mu_1)\le0$. By Proposition
\ref{positive}, for $t>0$,
\[
\int\limits_\Omega w^+(t) h^2dx \leq -\int\limits_0^t \int\limits_{\Omega}(G_1u_1 - G_2u_2)\ind_{\{w>0\}}d\g\, ds.
\]
However, $w>0$ implies $u_1>u_2$ and hence $G_1u_1\ge G_1u_2\ge
G_2u_2$. So the above yields $w^+=0$ and $u_1\leq u_2$.
\end{proof}

The next corollary is a straightforward consequence of Proposition~\ref{monotone}.

\begin{corollary}\label{uniq&bounds}
Let $\mu\in\mathcal{M}(\Omega)$, $G$ be a monotone homogeneous
Nemytskii operator. There exists at most one solution to the
problem $(\partial_t - \gen + G)u=0$, $u(0)=\mu$. The solution
satisfies the estimates
\[
-S\mu^- \le u \le S\mu^+
\]
and
\begin{equation}\label{energy}
\int_\tau^T \|u(s)\|^2_{\mathcal{F}}\,ds
\le \frac12\psi(\tau)\|\mu\|_{\mathcal{M}(\Omega)}^2.
\end{equation}
\end{corollary}
\begin{proof}
The first assertion is clear from Proposition~\ref{monotone}.
The second assertion follows from the comparison of the solution
to the problem $(\partial_t - \gen + G)u =0$, $u(0)=\mu$, with 
the solutions to the problems $(\partial_t - \gen +
G^{\mp})v^{\pm} =0$, $v^{\pm}(0)=\pm\mu^{\pm}$, where $G^{\mp}$
is the Nemytskii operator corresponding to the function $\ind_{[\mp
r\ge0]} g(x,r)$. Note that $v^{\pm}=\pm S\mu^{\pm}$ and that
$\ind_{[-r\ge0]}g(x,r)\le g(x,r) \le \ind_{[r\ge0]}g(x,r)$.
Hence the pointwise estimate follows.

Now we prove \eqref{energy}. For $\lambda>0$ let
$\zeta_\lambda(t):=S(\lambda)u(t)$, $t\in[0,T]$. Since $u\in
C\big([0,T];\, L^1\big)\cap L^2_{loc}\big((0,T);\,
\mathcal{F}\big)$, one has $\zeta_\lambda\in C\big([0,T];\,
L^1\cap C_b\big)\cap L^2_{loc}\big((0,T);\, \mathcal{F}\big)$.
Moreover, differentiating the equation $\zeta_\l(t)=S(\lambda+t)\mu-S(\lambda)(\convl Gu)(t)$, we obtain
\[
\big(\partial_t \zeta_\lambda\big)(t) = \gen S(\lambda)u(t) - S(\lambda)\big(Gu\big)(t),~t\in(0,T).
\]
Hence $\zeta_\lambda\in W^{1,1}_{loc}\big((0,T);\,
L^\infty\big)$. Since $|u(t)|\le S(t)|\mu|$, $t\in(0,T]$, we
conclude that $u\in L^\infty_{loc}\big((0,T);\, L^\infty\big)$.
Hence \eqref{wequ} holds with $\zeta=\zeta_\lambda$. For $\tau \in (0,T)$ we have
\[
\frac12\big\|\zeta_{\lambda/2}(T)\big\|_{L^2}^2 +
\int\limits_\tau^T \mathcal{E}\big(\zeta_{\lambda/2}(t)\big)dt
+ \int\limits_\tau^T\int\limits_\Omega \zeta_\lambda Gu\,d\gamma dt = \frac12\big\|\zeta_{\lambda/2}(\tau)\big\|_{L^2}^2.
\]
Passing to the limit as $\lambda\to0$, we arrive at
\[
\frac12\big\|u(T)\big\|_{L^2}^2 +
\int\limits_\tau^T \mathcal{E}\big(u(t)\big)dt
+ \int\limits_\tau^T\int\limits_\Omega u\, Gu\,d\gamma dt = \frac12\big\|u(\tau)\big\|_{L^2}^2.
\]
Finally, observe that $u\, Gu\ge0$ a.e. and that
$\big\|u(\tau)\big\|_{L^2}^2\le
\big\|S(\tau)|\mu|\big\|_{L^2}^2\le
\psi(\tau)\|\mu\|_{\mathcal{M}(\Omega)}^2.$
\end{proof}

\begin{proposition}\label{conv}
Let $g_n(x,r)\to g(x,r)$ for a.a. $x$ and locally uniformly in $r\in\mathbb{R}$, as
$n\to\infty$. Let $G_n, G$ be the corresponding monotone homogeneous  Nemytskii operators.
Let $\mu_n,\mu \in \mathcal{M}(\Omega)$ be such that $\mu_n\to\mu$ in the weak-$*$ topology of $\mathcal{M}(\Omega)$. 
In addition assume that
\begin{equation}\label{iintegral}
w:= \sup\limits_n \left[G_nS \mu^+_n - G_n(-S\mu^-_n)\right]\in L^1(Q).
\end{equation}
Let $u_n$ be the solutions to the problems
\begin{equation}
\label{e-n}
(\partial_t - \gen + G)u_n=0, \quad u_n(0)=\mu_n, \quad n\in
\mathbb{N}.
\end{equation}
Then $u_n\to u$ in $L^1(Q)$ as $n\to\infty$, and $u$ is
the solution to the problem $(\partial_t - \gen + G)u=0$,
$u(0)=\mu.$
\end{proposition}

\begin{proof}
First note that the sequence $(\mu_n)_n$ is bounded in
$\mathcal{M}(\Omega)$ since it is weak-$*$ convergent. Let
$M=\sup\limits_n\|\mu_n\|_{\mathcal{M}(\Omega)}<\infty$.
Now we have to pass to the limit in \eqref{e-n}.

Since $|G_nu_n|\le w \in L^1(Q)$ and by Proposition \ref{positive},
\[
\|G_n u_n\|_{L^1(Q)}\le M,
\]
the sequence $(G_nu_n)_n$ is a pre-compact set in the weak topology in $L^1(Q)$.
 By Proposition~\ref{SG-new}, $\convl$ is a completely
continuous operator on $L^1(Q)$.  Moreover, $S\mu_n\to S\mu$ by
Proposition~\ref{SG}. Therefore the sequence
$(\convl Gu_n)_n$, and hence the sequence $(u_n)_n$ are compact
in $L^1(Q)$. Moreover, due to~\eqref{energy}, $(u_n)_n$ is
weakly compact in
$L^2_{\mathrm{loc}}\big((0,T);\,\mathcal{F}\big)$.
 Let $(u_{n_l})$ be a sub-sequence of $(u_n)_n$ convergent,
in $L^1(Q)$ strongly, in
$L^2_{\mathrm{loc}}\big((0,T);\,\mathcal{F}\big)$ weakly and
a.e. on $Q$ to a limit $u$. Note that $|u_{n_l}(t)|\le
S(t)|\mu_{n_l}|\le M\psi(t)$ a.e. by
Corollary~\ref{uniq&bounds} and \eqref{MtoL}. Since for all
$t>0$ and a.a. $x\in \Omega$ one has $g_n(x,r)\to g(x,r)$ as
$n\to\infty$ uniformly in $r\in \left[-M\psi(t),
M\psi(t)\right]$, we conclude that $G_{n_l}u_{n_l}\to Gu$ as
$l\to\infty$ a.e. on $Q$. So $G_{n_l}u_{n_l}\to Gu$ in
$L^1(Q)$, by the Lebesgue dominated convergence theorem. Hence
we can pass to the limit in the equality $u_{n_l} = S\mu_{n_l}
- \convl G_{n_l}u_{n_l}$ as $l\to\infty$ and obtain that $u =
S\mu - \convl Gu$. Moreover, since $u_{n_l}\to u$ as
$l\to\infty$ weakly in
$L^2_{\mathrm{loc}}\big((0,T);\,\mathcal{F}\big)$, it follows
that $u$ satisfies \eqref{wequ}  with $f=-Gu$ for all $\zeta\in
W$. Hence $(\partial_t -\gen + G)u=0$ and $u(0)=\mu$. By
Corollary~\ref{uniq&bounds}, the solution to the latter
equation is unique so $(u_n)_n$ has a unique limit point $u$.
Hence $u_n\to u$ in $L^1(Q)$ strongly and in
$L^2_{\mathrm{loc}}\Big((0,T);\,\mathcal{F}\Big)$ weakly.
\end{proof}

The following is a straightforward consequence of Proposition~\ref{conv}.

\begin{corollary}
\label{cor5.4}
Let $G$ be a monotone homogeneous Nemytskii operator, $\mu_n\to\mu$ in weak-$*$ topology of  $\mathcal{M}(\Omega)$,
$\mu_n\ge0$, $\mathrm{supp}(\mu_n)\subset B_r$  and $\|\mu_n\|_{\mathcal{M}(\Omega)}\le c$.
Let $u_n$ be the solution to \eqref{e-n}.
Set $s_c(x,t):=c\sup\limits_{y\in B_r}p_t(x,y)$.
Assume that
\begin{equation}\label{iiintegral}
Gs_c\in L^1(Q).
\end{equation}
Then $u_n\to u$ in $L^1(Q)$ as $n\to\infty$, and $u$ is
the solution to the problem $(\partial_t - \gen + G)u=0$,
$u(0)=\mu$.
\end{corollary}

The next theorem is the main result of this section.

\begin{theorem}\label{exist-uniq}
Let \eqref{ultra} and \eqref{feller} hold. Let $\mu\in
\mathcal{M}(\Omega)$ satisfy the condition 
\begin{equation}\label{integral}
\iint\limits_Q \left[GS \mu^+ - G(-S \mu^-)\right] d\g dt<\infty.
\end{equation}
Then there exists a unique solution $u=u_\mu$
to the Cauchy problem
\begin{equation}\label{cauchy}
\begin{cases}
\left(\partial_t - \gen + G\right)u =0,\\
u(0)=\mu.
\end{cases}
\end{equation}
Moreover, $[u_\mu(t) - S(t)\mu]\to 0$ in $L^1$ as $t\to0$.
\end{theorem}

\begin{proof}
First we consider $g$ such that $\bar g\in L^1$ with
$\bar g(x):=\sup\limits_r |g(x,r)|$, $x\in \Omega$. Denote
$H(x,r):=\int_0^rg(x,s)ds$. Then $H$ is a convex positive
sub-linear function in $r$ for a.a. $x\in \Omega$. Consider the
functional
\[
J(u):= \frac12 \mathcal{E}(u) + \int\limits_\Omega H\left(x,u(x)\right)\gamma(dx),~ u\in\mathcal{F}.
\]
Then $\delta J = - \gen + G$. By \cite[Theorem III.4.1,~Proposition III.4.2]{sh97}, for $j=1,2,3,\ldots$ there exists a
unique solution $u_j\in L^2\Big((0,T);\,\mathcal{F}\Big)$ to
the Cauchy problem

\begin{equation}\label{bdd-nl}
\begin{cases}
(\partial_t - \gen + G)u_j =0,\\
u_j(0)=\mu_j\in L^1\cap L^\infty.
\end{cases}
\end{equation}
Moreover, $u_j\in L^\infty\left((0,T),\mathcal{F}\right)\cap
W^{1,2}\Big((0,T);\, L^2\Big)$.

If $\mu_j\to \mu$ as $j\to\infty$ in the sense of weak-$*$
convergence of measures, then, by Proposition~\ref{conv},
$u_j\to u$ as $j\to\infty$ in $L^1(Q)$, and $u$ is the unique
solution to \eqref{cauchy}. Indeed, we have to verify condition
\eqref{iintegral}. However,
\[
\sup\limits_n \left[G_nS \mu^+_n - G_n(-S \mu^-_n)\right]
\le \bar g\in L^1.
\]
Hence the assertion follows.

For a general $g$, let $E_k\subset\Omega$ be an increasing
sequence of subsets of finite measure such that $\Omega = \cup
E_k$. For $k=1,2,3,\ldots,$ let
$g_k(x,r):=\ind_{E_k}\mathrm{sgn}\left(g(x,r)\right)\left(\left|g(x,r)\right|\wedge
k\right)$, let $G_k$ be the corresponding Nemytskii operator and
let $u_k$ be the solution to the equation $(\partial_t - \gen +
G_k)u_k = 0$, $u_k(0)=\mu$ constructed above. Then, by
Corollary~\ref{uniq&bounds},  $-S{\mu^-}\leq u_k \leq
S{\mu^+}$, and hence
\begin{equation}\label{G-est}
  |G_ku_k|\leq |Gu_k|\leq
\begin{cases}
GS{\mu^+},& u_k \geq 0,\\
-G(-S{\mu^-}),& u_k<0,
\end{cases}
\quad \leq GS{\mu^+}-G(-S{\mu^-}).
\end{equation}
Since $GS{\mu^+}-G(-S{\mu^-})\in L^1(Q)$,
Proposition~\ref{conv} implies that $u_k\to u$ as $k\to\infty$
in $L^1(Q)$, and $u$ is the solution to \eqref{cauchy}.

To prove the last assertion, note that $S\mu - u = \convl Gu$.
So, by \eqref{G-est},
\[
\int\limits_\Omega |u(t)-S\mu(t)|d\g \leq \int\limits_0^t\int\limits_\Omega \left[GS{\mu^+}-G(-S{\mu^-})\right] d \g d\tau\to0
\mbox{ as }t\to0.
\]
\end{proof}

The next corollary together with the last assertion of the previous theorem provide the proof of assertions (d) and (e)
of Theorem~\ref{main1} for $\vark<\infty$.
\begin{corollary}\label{wss}
Let $0< p<1+\frac{2+\beta}{N+2\lambda}$. Then the problem
\[
\begin{cases}
\partial_tu - h^{-2}\div\big(h^2\nabla u) + r^{\beta}|u|^{p-1}u& \mbox{ in }\mathbb{R}^N, \\
u(0)=\vark\delta_0&
\end{cases}
\]
has a unique solution $u_\vark$ for every $\vark>0$. 
Conversely, for $1<p< 1+\frac{2+\beta}{N+2\lambda}$ and
$\vark\in(0,\infty)$, if $u$ is a solution to
\eqref{wparabolic} satisfying $u(t)\to \vark\delta_0$ as
$t\to0$ in the sense of weak-$*$ convergence of measures, then
$u=u_\vark$. 
\end{corollary}

\begin{proof}
In this case $\mathcal{E}(u)=\|\nabla u\|^2_{L^2_{h^2}}$ is the
bilinear quadratic form in $L^2_{h^2}$ with 
$C^1_c(\mathbb{R}^N)$ as its core. (It follows from Lemma
\ref{hardy} that $\Big(\mathcal{E}, C^1_c(\mathbb{R}^N)\Big)$
is closable in $L^2_{h^2}$.) Let $S$ denote the corresponding
semigroup and $k$ its integral kernel. By Theorem~\ref{hk}, $k$ obeys  the estimate \eqref{hkb}.
Now we verify the assumption of Theorem \ref{exist-uniq}:
\[
\begin{split}
\int\limits_0^T\int\limits_{\mathbb{R}^N}r^\beta|S(t)\vark\delta_0|^p & h^2dxdt
=  \vark^p \int\limits_0^T\int\limits_{\mathbb{R}^N}|x|^{\beta + 2\lambda}|k(t,x,0)|^pdxdt
\\
\le & \int\limits_0^T\int\limits_{\mathbb{R}^N}|x|^{\beta + 2\lambda}
c_\delta t^{-\frac {p(N+2\lambda)}2}e^{-\frac{p|x|^2}{4(1+\delta)t}}\left( \frac{|x|}{\sqrt t}
+ 1\right)^{-p\lambda} dxdt
\\
\le & c_{\delta,p}\int\limits_0^Tt^{\frac{\beta - (p-1)(N + 2\lambda)}2}dt
\int\limits_{\mathbb{R}^N}|\xi|^{\beta + 2\lambda}e^{-\frac{p|\xi|^2}{4(1+\delta)}}\left( |\xi|
+1\right)^{-p\lambda}d\xi.
\end{split}
\]
The integral in $t$ converges since $\frac{\beta + 2\lambda -
(p-1)N - 2p\lambda}2>-1$, that is, $p<1 +
\frac{2+\beta}{N+2\lambda}$. The integral in $\xi$ converges
since $\beta + 2\lambda + N= (\beta+2) + (2\lambda + N -2)>0$.

The second assertion follows from Corollaries~\ref{subcr-est}
and~\ref{uniq&bounds}
\end{proof}

\section{Very singular solutions}

In this section we construct a very singular solution to \eqref{wparabolic} and prove its uniqueness.
Throughout the section we assume that
\[
1<p<p^*=1+\frac{2+\b}{N+2\l}.
\]

We start this section by showing that every very singular solution (VSS) if it exists, dominates pointwise every source type solution (SS).
The next proposition is an analogue  of \cite[Lemma~1.3]{kpv89}.

\begin{proposition}
Let $v$ be a VSS  and $u$ be a SS to \eqref{wparabolic},
respectively. Then $u\le v$ pointwise for a.a. $(x,t)\in \R^N\times(0,T)$.
\end{proposition}
\begin{proof}
Let $\int_{B_1}u(t)h^2dx\to \vark<\infty$ as $t\to 0$. Let
$\tau_0>0$ be such that $\int_{B_1}v(t)h^2dx>\vark$ for all $0<t\le
\tau_0$. Then, for $\tau\in (0,\tau_0)$ there exists
$\var_\tau \in L^1_{h^2}$ such that $0\le \var_\tau\le
v(\tau)\ind_{B_1}$ and $\|\var_\tau\|_{L^1_{h^2}}=\vark$.

Let $u^{(\tau)}$ be the solution to the problem
\[
(\partial_t-\Delta_{h^2})u+r^\b u^p=0, \quad u(0)=\var_\tau.
\]
Thanks to Proposition~\ref{local-quanti} it is easy to check that
\[
v\in L^1(\R^N\times(\tau,t), h^2dxdt)\cap L^p(\R^N\times(\tau,t), r^\b h^2dxdt).
\]
Then by Proposition~\ref{monotone}
\begin{equation}
\label{utau}
u^{(\tau)}(t)\le v(t+\tau), \ \ t>0.
\end{equation}

Since $\|u^{(\tau)}(0)\|_{L^1_{h^2}}=\vark$ and ${\rm supp}\,u^{(\tau)}(0)\subset {\rm supp}\, v(\tau)$, it follows that
$u^{(\tau)}(0)h^2dx \to \vark \delta_0$ in weak-$*$ topology  of ${\cal M}(\Omega)$.
Hence $u^{(\tau)} \to u_\vark$ in $L^1_{h^2}(Q)$ by Corollary~\ref{cor5.4}, where \eqref{hkb} is used to verify \eqref{iiintegral}. Then \eqref{utau}
implies that $u_\vark\le v$.
\end{proof}

The above leads to an immediate construction of the minimal VSS.

\begin{corollary}
$u_\infty:=\lim\limits_{\vark\to\infty} u_\vark$ is the minimal VSS, where $u_\vark$ is the solution from Corollary~\ref{wss}.
\end{corollary}
\begin{proof}
Using Proposition~\ref{local-quanti} one can easily verify that the above limit exists and is a solution to \eqref{wparabolic}, \eqref{null}.
\end{proof}

In the next proposition we follow the construction from \cite[Theorem~4.1]{kv92}.
\begin{proposition}
$U_\infty(x,t):=\sup\{u(x,t)\,:\,u\ \ \text{is a positive singular
solution}\}$ is the maximal VSS.
\end{proposition}

\begin{proof}
Let $u$ be a solution for \eqref{wparabolic}, \eqref{null}.
It follows from Lemma~\ref{local-quali} and Proposition~\ref{local-quanti} that, for all $R>0$ one has $u\in C^{2,1}(\mathbb{R}^N\setminus B_R\times [0,T))$ and
$u(x,t)\to0$ as $|x|\to\infty$ uniformly in $t$. Moreover, by Proposition~\ref{local-quanti}
\begin{equation}\label{KO}
    u(x,t)\le {c}{(|x|^2+t)^{-\frac{2+\b}{2(p-1)}}}, \quad (x,t)\in {\mathbb R}^N\times (0,\infty),
\end{equation}
with a constant $c>0$ independent of $u$. Let $v$  be the solution of the linear inhomogeneous problem
\[
\begin{cases}
(\partial_t - \Delta_{h^2})v=0 & \mbox{ in } \mathbb{R}^N\setminus B_R \times (0,\infty),
\\
v(x,0)=0,& x\in \mathbb{R}^N\setminus B_R,
\\
v(x,t)= cR^{-\frac{2+\beta}{p-1}}, & x\in \partial B_R, t>0.
\end{cases}
\]
Then, by the maximum principle, $u\le v$.
Note that $v$ enjoys the estimate

\begin{equation}\label{kern-decay}
    0\le v(x,t)\le C_R\int\limits_0^t\int\limits_{[R<|x|<2R]}p_s^{h^2}(x,y)dy\,ds,
\end{equation}
where $C_R>0$ is a constant and $p^{h^2}_t(\cdot,\cdot)$ is the fundamental solution to the
linear equation $(\partial_t - \Delta_{h^2})u=0$.
Hence all $u$ and $U_\infty$ satisfy \eqref{kern-decay} with $v$ replaced by $u$ and $U_\infty$, respectively.
Note also that $U_\infty$ satisfies the estimate \eqref{KO} with $u$ replaced by $U_\infty$.
In particular, $U_\infty(\tau)\in L^1_{h^2}$ and $U_\infty \in L^1(\R^N\times(\tau,T),\,h^2dxdt)\cap L^p(\R^N\times(\tau,T),\,r^\b h^2dxdt)$
for all $\tau>0$.

By Theorem~\ref{exist-uniq} for $t>\tau$ the problem
\begin{equation*}
\left\{
\begin{array}{c}
(\partial_t-\Delta_{h^2})u+r^\b u^p=0, \quad t>\tau,
\\
u(\tau)=U_\infty(\tau)
\end{array}
\right.
\end{equation*}
has a unique solution $u^{(\tau)}$.

For every singular solution $u$ we have that $u^{(\tau)}(\tau)\ge
u(\tau)$. Therefore by Proposition~\ref{monotone} $u^{(\tau)}(t)\ge
u(t)$ , and hence $u^{(\tau)}(t)\ge
U_\infty(t)$ for all $t\ge \tau$. Moreover, for $\tau'\le \tau$ one has
\[
u^{(\tau')}(\tau)\ge U_\infty(\tau)=u^{(\tau)}(\tau).
\]
Using Proposition~\ref{monotone} again we obtain that
\[
u^{(\tau')}(t)\ge u^{(\tau)}(t), \quad \tau'\le \tau\le t.
\]
By Proposition~\ref{local-quanti} it follows that, for all $t_0>0$
and $\tau<\frac{t_0}2$, with $\rho:=\sqrt{t_0-\tau}\ge \sqrt{\frac{t_0}2}$
\[
\iint\limits_{t>t_0}|\nabla u^{(\tau)}|^2 h^2dx\,dt\le c\left(\int\limits_{|x|>\rho}U_\infty(x,\tau)^2h^2dx + \rho^{N+2\lambda -2\frac{2+\beta}{p-1}}\right)
\le c t_0^{\frac N2+\lambda -\frac{2+\beta}{p-1}}.
\]
So $(\nabla u^{(\tau)})_\tau$ is
bounded in $L^2_{loc}({\mathbb R}^N\times (0,\infty),h^2dx\,dt)$
uniformly in $\tau$. Hence $u_\tau\uparrow u$ as $\tau
\downarrow 0$. Now passing to the limit in $\tau$ it is easy
to see that
\[
(\partial_t-\Delta_{h^2})u+r^\b u^p=0.
\]
Furthermore, by \eqref{kern-decay} for $x\not=0$ we have that
$u(x,t)\to 0$ as $t\to 0$. Thus $u$ is a singular solution, and
hence $u\le U_\infty$. Since $u\ge u_\tau\ge U_\infty$, we
conclude that $u=U_\infty$.
\end{proof}

\begin{lemma}
The minimal VSS $u_\infty$ and and the maximal VSS $U_\infty$ to \eqref{wparabolic} are self-similar.
More precisely,
\[
u_\infty(x,t)=t^{-\sigma} v_\infty(\frac{x}{\sqrt{t}}), \quad U_\infty(x,t)=t^{-\sigma} V_\infty(\frac{x}{\sqrt{t}})
\]
where $v_\infty$ and $V_\infty$ are positive solutions to the
problem
\begin{equation}\label{self-sim}
    \begin{cases}
-K^{-1}\n (K\n v)-\sigma v +r^\b v|v|^{p-1}=0, \\
r^{2\sigma}v\to 0\ \ \text{as}\ r\to\infty,
    \end{cases}
\end{equation}
with $\sigma=\frac{\b+2}{2(p-1)}$ and
$K=r^{2\l}e^{\frac{r^2}{4}}$.
\end{lemma}

\begin{proof}
Let $u$ be a singular solution to \eqref{wparabolic}. Then
$T_\rho u$ defined by $T_\rho
u(x,t):=\rho^{\frac{\b+2}{p-1}}u(\rho x, \rho^2 t)$, is another
singular solution to \eqref{wparabolic}. Moreover, $T_\rho
u_c=u_{c_\rho}$ with $c_\rho =c \rho^{\frac{\b+2}{p-1}-N-2\l}$.
Hence by definition $T_\rho U_\infty=U_\infty$ and  $T_\rho
u_\infty=u_\infty$. Now the assertion follows with
$\rho=t^{-\frac12}$.
\end{proof}

\begin{proposition}
A VSS to \eqref{wparabolic} is unique, i.e. 
$U_\infty=u_\infty$.
\end{proposition}

\begin{proof}
It suffices to show that $V_\infty\le v_\infty$ . Let
$w:=V_\infty - v_\infty $. Note that $w$ is a sub-solution
to the equation
\[
-K^{-1}\n (K\n w)-\sigma w +r^\b v_\infty^{p-1}w=0.
\]
Since $-K^{-1}\n (K\n v_\infty)-\sigma v_\infty +r^\b v_\infty^p=0$, it follows
from \cite[Theorem 4.1]{Da} that
\[
\int|\n \theta|^2 K\,dx + \int \left(r^\b v_\infty^{p-1} - \sigma\right)\theta^2K\,dx\ge0 \mbox{ for all }\theta\in C^{0,1}_c(\mathbb{R}^N).
\]
Since $r^{2\sigma}V_\infty\to 0$ as $r\to\infty$ it follows that, for a sufficiently large $R$,  one has $w(x)\le R^{-2\sigma}$ for all $x$, $|x|>R$.
By the weak maximum principle, we infer that $w(x)\le R^{-2\sigma}$ for all $x$, $|x|<R$. Hence $w\le0$.
\end{proof}

As the positive solution to \eqref{self-sim} produces a VSS, it is
clear that \eqref{self-sim} has a unique positive solution.  To find
it
one can use a
variational approach almost identical to that in \cite{ek87}.
Namely, one considers
the nonlinear
functional $J$ on the Banach space $X:=H^1_K(\mathbb{R}^N)\cap
L^{p+1}(\mathbb{R}^N, r^{\beta}Kdx)$,
\begin{equation}
\label{functional}
J(\theta):=\frac12\int_{\mathbb{R}^N} |\nabla \theta|^2 K\,d\xi+\frac1{p+1}\int_{\mathbb{R}^N} |\theta|^{p+1}r^{\beta}K\,d\xi
-\frac\mu2\int_{\mathbb{R}^N} | \theta|^2K\,d\xi.
\end{equation}

To show that $J$ is bounded below we need the following auxiliary result.

\begin{lemma}\label{bbound} For any $\varepsilon>0$ there
exists $C_\varepsilon>0$ such that
\[\|\theta\|_{L^2_K}^2\leq \varepsilon\left(\|\nabla\theta\|_{L^2_K}^2 + \|\theta\|_{L^2_{r^\beta
K}}^{p+1}\right) + C_\varepsilon,~\theta\in X.
\]
\end{lemma}
 \begin{proof}
For $\varepsilon>0$, choose $R_\varepsilon>r_\varepsilon>0$
such that
\[
\frac{16}{R_\varepsilon^2} + \frac{4r_\varepsilon^2}{(N-2+2\lambda)^2}<\frac\varepsilon2.
\]
Then, by Lemma \ref{Hardy1}, we have
\[
\int\limits_{B_{r_\varepsilon}\cup \mathbb{R}^N\setminus B_{R_\varepsilon}}|\theta|^2K\,dx \leq \frac\varepsilon2 \|\nabla\theta\|_{L^2_K}^2.
\]
Now, by the Young inequality,
\[
\int\limits_{B_{R_\varepsilon}\setminus B_{r_\varepsilon}}|\theta|^2K\,dx \leq
\frac\varepsilon2 \int\limits_{B_{R_\varepsilon}\setminus B_{r_\varepsilon}}|\theta|^{p+1}K\,dx
+ C_{p,\varepsilon} \int\limits_{B_{R_\varepsilon}\setminus B_{r_\varepsilon}}K\,dx.
\]
 \end{proof}

Now we are ready to show the existence of a non-trivial
minimizer of the functional $J$. 
This follows immediately from the next proposition.

\begin{proposition}
\label{minimizer1} The functional $J$ defined in \eqref{functional}
is bounded below and
lower semi-continuous with respect to the weak topology. Moreover, $J(\theta)\to \infty$ as $\|\theta\|_X\to
\infty$. 
\end{proposition}

\begin{proof} Let $\theta_n\to\theta$ weakly in $X$. Then
$\theta_n\to\theta$ strongly in $L^2_K$ since
$H^1_K\stackrel{compact}{\hookrightarrow}L^2_K$, by
Corollary~\ref{compact}. So $\liminf\|\theta_n\|_{H^1_K}\geq
\|\theta\|_{H^1_K}$, and due to the lower semi-continuity of the
$L^p$-norm w.r.t. the weak convergence (see, e.g. \cite{LL})
$\liminf\|\theta_n\|_{L^{p+1}_{r^\beta K}}\geq
\|\theta\|_{L^{p+1}_{r^\beta K}}$ and $\lim\|\phi_n\|_{L^2_K}=
\|\phi\|_{L^2_K}.$ The last two assertions follow directly from
Lemma~\ref{bbound}.
\end{proof}

Next we show that the minimizer is nontrivial and can be chosen non-negative.

\begin{proposition}
\label{minimizer2} Let $\mu> \frac{N+2\l}{2}$. Then there
exists a non-trivial minimizer of $J$ which can be chosen
nonnegative.
\end{proposition}
\proof Note that $J(0)=0$. Let $\tau>0$. Set  $\phi=\tau
e^{-\frac{r^2}{4}}$. Then
\[
E(\tau \phi)=(\frac{N+2\l}{2}-\mu)\tau^2\int \phi^2 K\,dx +\tau^{p+1} \int \phi^{p+1} r^{\beta}K\,dx.
\]
Now it clear that there exists $\tau>0$ such that $J(\tau \phi)<0$,
hence zero is not a minimizer. The last assertion follows from the
fact that $J(\theta)=J(|\theta|)$. \qed

The minimizer is exponentially decaying at infinity, which is shown in the next proposition.

\begin{proposition}
\label{exp-decay} Let $v\in H^1_K(\mathbb{R}^N)\cap
L^{p+1}(\R^N, r^{\beta}Kdx)$ be a solution to
\begin{equation}
\label{stat1}
-K^{-1}\nabla (K \nabla v)-\mu v +r^{\beta}v|v|^{p-1}=0.
\end{equation}
There exists $C>0$ such that
\[
|v|\le C e^{-\frac{r^2}{8}}, \ \text{on}\ \mathbb{R}^N\setminus B_1.
\]
\end{proposition}
\proof We follow \cite{ek87} simplifying the arguments. Let
$w:=v e^{|x|^2/8}$. Then $w$ satisfies the equation
\begin{equation}
\label{eqw}
-\frac1{h^2}\n(h^2 \n w)+Vw=0
\end{equation}
with
$V:=\frac{N+2\l}{4}-\mu+\frac{r^2}{16}+e^{-\frac{(p-1)r^2}{8}}r^{\beta}|w|^{p-1}$.
One can choose $R>0$ such that $V(x)\ge 0$ on $|x|>R$. It is
easily seen that solutions to \eqref{eqw} are locally bounded
outside the point $x=0$.  Let $M:=\sup\limits_{1<|x|<R}w(x)$.
Looking at \eqref{eqw} on $|x|>1$ and taking $\var:=(w-M)^+$ as
a test function we obtain that $\var=0$. Changing $w$ by $-w$
in \eqref{eqw} we see that $|w(x)|\le M$, which proves the
assertion.\qed
%


\renewcommand\thesection{\Alph{section}}
\setcounter{section}{0}

\section{Appendix: Hardy-type inequality and compact embedding}

\begin{lemma}
\label{Hardy1} For $\lambda>\frac{2-N}2$, $\alpha\geq0$, 
$K=r^{2\lambda}e^{\alpha r^2}$, $\theta\in C_c^\infty(\R^N)$, there holds 
\begin{equation}\label{hardy}
\int_{\R^N} |\nabla \theta|^2 Kdx\geq
\int_{\R^N}\left(\alpha^2{r^2}+\alpha(N+2\lambda) + \left(\frac{N-2 +2\lambda}2\right)^2\frac1{r^2}\right)|\theta|^2 Kdx.
\end{equation}
\end{lemma}
\begin{proof} First, notice that ${\rm div}(x|x|^q)=(N+q)|x|^q$ for all $q>-N$. Now let
$v=r^\lambda e^{\frac\alpha2 r^2}\theta$. Then $v\in H^1(\R^N)$ and
\[
\nabla v=r^\lambda e^{\frac\alpha2|x|^2}\nabla \theta+ x\left(\alpha + \frac\lambda{r^2}\right) v.
\] Hence we have
\[
\begin{split}
\int_{\R^N} |\nabla \theta|^2 Kdx & =
\int_{\R^N} |\nabla v|^2dx +\int_{\R^N} \left(\alpha^2r^2 +  2\alpha\lambda + \frac{\lambda^2}{r^2} \right)v^2dx-
\int_{\R^N} (\nabla v^2)\cdot x\left(\alpha + \frac\lambda{r^2}\right)dx\\
& =
\int_{\R^N} |\nabla v|^2dx +\int_{\R^N} \left(\alpha^2r^2  +  \alpha( N + 2\lambda) +
\frac{\lambda^2 + (N-2)\lambda}{r^2} \right)v^2dx.
\end{split}
\]
Now the assertion follows from the standard Hardy inequality.
\end{proof}

\begin{corollary}\label{compact} Let $\lambda>
-\frac{N-2}2$, $\alpha>0$,  $K=e^{\alpha
r^2}r^{2\lambda}$. Then $H^1(\mathbb{R}^N,Kdx)$ is compactly
embedded into $L^2(\mathbb{R}^N,Kdx).$
\end{corollary}

\begin{proof}
It suffices to prove that, given $v_n\to 0$ weakly in
$H^1(\mathbb{R}^N,Kdx)$, the sequence $(v_n)_n$ converges to 0
strongly in $L^2(\mathbb{R}^N, Kdx)$.

Let $m:=\sup\limits_n\|v_n\|_{H^1}$. We use the following
decomposition: for $0<r_0<R_0$,
\begin{equation}\label{dec}
    \int_{\mathbb{R}^N}v^2Kdx = \int_{r<r_0}v^2 Kdx + \int_{r_0<r<R_0}v^2Kdx
+ \int_{r>R_0}v^2Kdx.
\end{equation}
Fix $\varepsilon>0$ and choose $r_0$ and $R_0$ such that
\[
2mr_0^{2}\left(\frac{N-2 + 2\lambda}2\right)^{-2}<\varepsilon \mbox{ and }
\frac{2m}{\alpha^2 R_0^{2}}<\varepsilon.
\]
Then by Lemma~\ref{Hardy1}
\[
\begin{split}
\int_{r<r_0}v^2  Kdx + \int_{r>R_0}v^2 Kdx
\leq r_0^2\int_{\mathbb{R}^N}\frac{v^2}{r^2} Kdx + \frac1{R_0^2}\int_{\mathbb{R}^N}r^2v^2Kdx\\
\leq\left(r_0^2\left(\frac{N-2 + 2\lambda}2\right)^{-2} + \frac{16}{R_0^2}\right)m<\varepsilon.
\end{split}
\]
Finally, the sequence $(v_n)$ is bounded in $H^1(B_{R_0}\setminus
B_{r_0},Kdx)=H^1(B_{R_0}\setminus B_{r_0},dx)$. So $v_n\to0$
weakly in $H^1(B_{R_0}\setminus B_{r_0},dx)$,  hence $v_n\to0$
strongly in $L^2(B_{R_0}\setminus
B_{r_0},dx)=L^2(B_{R_0}\setminus B_{r_0},Kdx)$. Thus, for all
$\varepsilon>0$
\[
\limsup\limits_{n\to\infty}\int_{\mathbb{R}^N}v^2_nKdx<\varepsilon.
\]
The same argument implies the second assertion.
\end{proof}

\noindent {\bf Acknowledgement.} The authors would like to
acknowledge support of the Royal Society through the
International Joint Project Grant 2006/R1.

\begin{small}

\end{small}

\end{document}